\documentclass{article}

\usepackage{amsmath,amssymb,amsfonts,amsthm}
\usepackage{color,graphicx,bbm,verbatim}
\usepackage{mathtools}
\usepackage{epstopdf}
\usepackage{algorithmic}
\usepackage{subfig}
\usepackage{booktabs}
\usepackage{dcolumn}
\usepackage[detect-all]{siunitx}
\usepackage{hyperref}
\usepackage[a4paper,tmargin=1.45in,bmargin=1.45in]{geometry}
\usepackage{enumitem}
\usepackage[capitalize]{cleveref}

\ifpdf
  \DeclareGraphicsExtensions{.eps,.pdf,.png,.jpg}
\else
  \DeclareGraphicsExtensions{.eps}
\fi

\newtheorem{theorem}{Theorem}[section] 
\newtheorem{lemma}[theorem]{Lemma}
\newtheorem{remark}[theorem]{Remark}

\numberwithin{equation}{section}

\def\matlab{MATLAB}

\def\R{\mathbb{R}}
\def\C{\mathbb{C}}

\def\bigO{\mathcal{O}}
\def\imagunit{\mathbf{i}}

\DeclareMathOperator{\Arg}{Arg}

\DeclareMathOperator*{\argmin}{arg\,min}

\DeclareMathOperator{\bd}{\mathrm{bd}}
\DeclareMathOperator{\interior}{\mathrm{int}}

\def\e{\mathrm{e}}
\def\eit{\eix{\theta}}
\def\emit{\emix{\theta}}

\newcommand{\eix}[1]{\e^{\imagunit #1}}
\newcommand{\emix}[1]{\e^{-\imagunit #1}}

\def\eps{\varepsilon}
\def\smin{\sigma_{\min}}

\def\aleps{\alpha_\eps}
\def\rhoeps{\rho_\eps}
\def\dtu{\tau(A,B)}

\def\Kcal{\mathcal{K}}
\def\Kcon{\Kcal(A)}
\def\Kinv{\Kcon^{-1}}
\def\n{n}

\def\sffn{f_\gamma}
\def\sfset{\mathcal{F}}
\def\sgfn{g_\gamma}
\def\sgset{\mathcal{G}}
\def\shfn{h_\gamma}
\def\shset{\mathcal{H}}

\def\costh{\cos \theta}

\renewcommand{\Re}{\mathrm{Re}\,}
\renewcommand{\Im}{\mathrm{Im}\,}

\def\beq{\begin{equation}}
\def\eeq{\end{equation}}
\def\beqs{\begin{equation*}}
\def\eeqs{\end{equation*}}
\def\bseq{\begin{subequations}}
\def\eseq{\end{subequations}}

\usepackage{caption}
\usepackage{newfloat}
\usepackage[section]{algorithm}
\usepackage{algorithmic}
\DeclareFloatingEnvironment[placement=htb]{algfloat}
\newcommand{\algnote}[1]{\footnotesize \sc{Note: \it#1 } }

\title{Fast Interpolation-based Globality Certificates for\\Computing Kreiss Constants and the\\Distance to Uncontrollability}
\author{Tim Mitchell\thanks{
Max Planck Institute for Dynamics of Complex Technical Systems, Magdeburg, 39106 Germany \href{mailto:mitchell@mpi-magdeburg.mpg.de}{\texttt{mitchell@mpi-magdeburg.mpg.de}}.
The author's visits to the 
Courant Institute of Mathematical Sciences, 
New York University were supported by the U.S. National Science Foundation grant DMS-1620083}
}
\date{\small Submitted: October 2nd, 2019\\ Revised: August 9th, 2020, November 11th, 2020}

\begin{document}
\maketitle
\vspace{-0.5cm}
\begin{abstract}
We propose a new approach to computing global minimizers
of singular value functions in two real variables.
Specifically, we present new algorithms to compute the Kreiss constant of a matrix
and the distance to uncontrollability of a linear control system, both to arbitrary accuracy.
Previous state-of-the-art methods for these two quantities rely on 2D level-set tests 
that are based on solving large eigenvalue problems.
Consequently, these methods are costly, i.e., $\bigO(n^6)$ work using dense eigensolvers,
and often multiple tests are needed before convergence.
Divide-and-conquer techniques have been proposed that reduce the work complexity to
$\bigO(n^4)$ on average and $\bigO(n^5)$ in the worst case, but
these variants are nevertheless still very expensive and can be numerically unreliable.
In contrast, our new interpolation-based globality certificates 
perform level-set tests by building interpolant approximations to certain one-variable continuous functions
that are both relatively cheap and numerically robust to evaluate.
Our new approach has a $\bigO(kn^3)$ work complexity and uses $\bigO(n^2)$ memory, 
where $k$ is the number of function evaluations necessary to build the interpolants.
Not only is this interpolation process mostly ``embarrassingly parallel," 
but also low-fidelity approximations typically suffice for all but the final interpolant, which must be built to high accuracy.  
Even without taking advantage of the aforementioned parallelism, $k$ is sufficiently small
that our new approach is generally orders of magnitude faster than the previous state-of-the-art.
\end{abstract}


\medskip
\noindent
\textbf{Notation:} $\| \cdot \|$ denotes the spectral norm, $\smin(\cdot)$ the smallest singular value, 
$\Lambda(\cdot)$ the spectrum, and
$\kappa(\cdot)$ the condition number of a matrix with respect to the spectral norm.
A matrix pencil $A - \lambda B$ and its spectrum
are denoted $(A,B)$ and  $\Lambda(A,B)$, respectively,
and $(A,B)$ is a regular matrix pencil if there exists at least one $\lambda \in \C$
such that $\det (A - \lambda B) \neq 0$.
A matrix $A \in \C^{2\n\times 2\n}$ is (skew-)Hamiltonian if $(JA)^* = JA$ ($A^*J = JA$),
where \mbox{$J = \begin{bsmallmatrix} 0 & I \\ -I & 0 \end{bsmallmatrix}$}.
A matrix pencil $(A,B)$ with $A,B \in \C^{2\n\times 2\n}$ 
is skew-Hamiltonian-Hamiltonian (sHH) if $B$ is skew-Hamiltonian and $A$ is Hamiltonian.
Euler's number, $2.71828\ldots$, is denoted $\mathrm{e}$, 
while $\bd \mathcal{A}$ and $\interior \mathcal{A}$
are the boundary and interior of a set $\mathcal{A}$, respectively.

\section{Introduction}
We begin with two important quantities that can be written as global optimization
problems of certain singular value functions in two real variables and describe existing algorithms
for computing these quantities and their limitations.
In this discussion, we also establish necessary background and context for 
understanding our new approach.
\subsection{The distance to uncontrollability}
\label{sec:intro_dtu}
Consider the linear control system
\beq
	\label{eq:AxBu}
	\dot x = Ax + Bu,
\eeq
where $A \in \C^{\n \times \n}$, $B \in \C^{\n \times m}$, and 
the state $x \in \C^\n$ and control input $u \in \C^m$ are dependent on time.
System \eqref{eq:AxBu} is \emph{controllable} if, for any pair of initial and final states $x(0)$ and $x(T)$, $T > 0$,
there exists a control function $u(\cdot)$ that takes $x(0)$~to~$x(T)$. 
However, a more robust measure is the \emph{distance to uncontrollability},
which was introduced in \cite{Pai81} and we denote by $\dtu$.
Given a system \eqref{eq:AxBu}, $\dtu$ specifies the distance to the nearest matrix pair $A_\mathrm{uc},B_\mathrm{uc}$
such that \mbox{$\dot x = A_\mathrm{uc}x + B_\mathrm{uc}u$} is uncontrollable,
with the distance being zero if \eqref{eq:AxBu} is uncontrollable and positive otherwise.\footnote{
A new alternative for assessing the distance to uncontrollability
of systems with input \emph{and output}
was recently proposed in \cite{MarFG18,FazGM19}, but here we focus on the standard definition.}
In~\cite{Eis84}, Eising showed that $\dtu$ is equal to the globally minimal value of the following singular value function 
\beq
	\label{eq:dtu}
	\dtu = \min_{z \in \C} \smin \left( \begin{bmatrix} A - zI & B\end{bmatrix} \right).
\eeq

While many methods have been proposed for $\dtu$ over the years (see \cite[p.~990]{Gu00} for a historical overview), 
the first polynomial-time algorithm to correctly estimate $\dtu$ to within a constant factor
is due to Gu \cite{Gu00}.
The basis of Gu's algorithm is a 2D level-set test.
Given a guess $\gamma \geq 0$ for the value of $\dtu$ and a parameter $\eta > 0$,
Gu's 2D level-set test verifies whether or not there exists one or more points $\tilde z \in \C$ such that 
\beq
	\label{eq:gu_test}
	\gamma =
 	\smin \left( \begin{bmatrix} A - \tilde zI & B\end{bmatrix} \right)
	= 
	\smin \left( \begin{bmatrix} A - (\tilde z+\eta)I & B\end{bmatrix} \right)
\eeq
holds.
Moreover, if such points exist, Gu's test also returns their values.
Clearly $\dtu \leq \gamma$ holds if \eqref{eq:gu_test} is satisfied by some $\tilde z$.
However, if there are no such points, \cite[Theorem~3.1]{Gu00}
states that the following lower bound must instead hold:
\beq
	\label{eq:dtu_lb}
	\dtu > \gamma - \tfrac{\eta}{2}.
\eeq	
By starting with $\gamma_0 \geq \dtu$
and using $\gamma_{k+1} \coloneqq \tfrac{1}{2}\gamma_k$ and $\eta_{k+1} \coloneqq \gamma_{k+1}$,
Gu's method for $\dtu$ repeatedly applies his 2D level-set test 
until it no longer finds any points $\tilde z \in \C$ satisfying \eqref{eq:gu_test}.
Assuming exact arithmetic is used, at termination, the last estimate is guaranteed to be within a factor of two of $\dtu$ (in either direction).
The cost of Gu's method is dominated by the need to compute 
all real eigenvalues of a large \emph{generalized} eigenvalue problem of order $2n^2$ 
for each 2D level-set test, as described in \cite[Section~3.2]{Gu00}.
Consequently, when using standard dense eigensolvers, such as those based on the QZ algorithm, 
Gu's estimation algorithm has a $\bigO(n^6)$ work complexity\footnote{As in
\cite{BurLO04},
work complexities are given in terms of 
considering all computations of singular values, eigenvalues, etc., as \emph{atomic operations} 
with cubic costs in the dimensions of the associated matrices,
and we further assume that these costs reduce to linear if sparse methods are applicable.
}
and requires $\bigO(n^4)$ memory.
In practice, this limits Gu's method to all but the smallest of problems,
and furthermore, in inexact arithmetic, his 2D level-set test can fail due to rounding errors,
particularly as $\eta$ gets closer and closer to zero; see \cite[p.~358]{BurLO04}.

Using Gu's 2D level-set test, 
Burke, Lewis, and Overton \cite{BurLO04} proposed the first two algorithms to compute 
$\dtu$ to arbitrary accuracy, again assuming exact arithmetic.
Their first method \cite[Algorithm~5.2]{BurLO04} uses Gu's test in a trisection iteration 
in an effort to minimize the speed at which $\eta\to0$ as trisection converges to $\dtu$.
In turn, this helps reduce the chances of Gu's test failing numerically before the 
estimate to $\dtu$ has been sufficiently resolved.
We forgo describing trisection in detail, but mention that trisection is not a panacea,
since if $\dtu$ is very small, so must $\eta$ be to resolve $\dtu$ to even a single digit of accuracy; 
see \cite[Lemma~B.1 and Corollary~B.2]{Mit19}.
The authors' second method \cite[Algorithm~5.3]{BurLO04} 
combines Gu's test and local optimization to yield an optimization-with-restarts iteration.
As $\smin \left( \begin{bsmallmatrix} A - zI & B\end{bsmallmatrix} \right)$ is 
Lipschitz and will generally be smooth, even at minimizers, finding locally optimally approximations
to $\dtu$ via standard optimization techniques is straightforward and can be done with few function evaluations. 
Computing the value, gradient, and Hessian of 
$\smin \left( \begin{bsmallmatrix} A - zI & B\end{bsmallmatrix} \right)$ is $\bigO(n^3)$ work via standard dense SVD methods,
while the function value and gradient can be obtained in just $\bigO(n)$ work via sparse SVD methods.
Thus, given a local minimizer $z_k$ of \eqref{eq:dtu} with 
$f_k = \smin \left( \begin{bsmallmatrix} A - z_kI & B\end{bsmallmatrix} \right)$,  
\cite[Algorithm~5.3]{BurLO04} uses Gu's test with carefully chosen values 
for parameters $\gamma$ and $\eta$ such that the algorithm terminates if $f_k$ is sufficiently close to $\dtu$
in a relative sense.  If not, 
Gu's test still provides new level-set points such that running optimization from them must
yield a better (lower) minimizer.
Since the objective function in \eqref{eq:dtu} is semi-algebraic, 
it has a finite number of locally minimal function values; see \cite[p.~359]{BurLO04}.
As a result, \cite[Algorithm~5.3]{BurLO04} must terminate at a globally optimal minimizer
within a finite number of optimization restarts (for brevity, throughout the paper we assume that 
optimization finds stationary points exactly).
In practice, the number of restarts is typically only a handful, which makes it many times faster
than the trisection iteration.

Shortly thereafter, \cite{GuMOetal06} showed how the large generalized eigenvalue problem in Gu's 2D level-set test 
can be reduced to a \emph{standard} eigenvalue problem (but still of order $2n^2$),
and then proposed a divide-and-conquer technique to compute the relevant eigenvalues
in $\bigO(n^4)$ work on average and $\bigO(n^5)$ in the worst case.
While divide-and-conquer enables asymptotically faster versions of all of the methods of \cite{Gu00,BurLO04} described above, it does not address the aforementioned numerical issues inherent 
in Gu's 2D level-set test.  In fact, divide-and-conquer introduces additional numerical uncertainties,
as it relies on sparse shift-and-invert eigensolver techniques.
As we mentioned in \cite[Section~8]{Mit19}, one issue is that divide-and-conquer assumes that
such eigensolvers always return the closest eigenvalues to a given shift, which while reasonable, is 
not always true in practice.  Furthermore, sparse eigensolvers such as \texttt{eigs} in \matlab\ 
can have convergence issues when the norm of the matrix in question gets large; see \cite[p.~500]{GuMOetal06}.
Nevertheless, it is not always clear whether divide-and-conquer will be more or less reliable
than Gu's original test using dense eigensolvers, and 
indeed the experiments of \cite[Section~4]{GuMOetal06}
do demonstrate that divide-and-conquer can be a much faster and effective alternative for computing $\dtu$.
For a thorough discussion of the numerical difficulties of both Gu's original approach
and divide-and-conquer, see \cite[Sections~4.1, 5.2, and 5.3]{GuMOetal06} and the references within.

Finally, we recently showed how the numerical reliability 
of all of these $\dtu$ methods can be greatly improved via a crucial 
reinterpretation and modified version of Gu's 2D level-set test; see \cite[Key~Remark~6.3]{Mit19}.

\subsection{The Kreiss constant of a matrix}
\label{sec:intro_kreiss}
We now turn to another important quantity, namely, the \emph{Kreiss constant} of a matrix $A \in \C^{n \times n}$,
which comes in continuous- and discrete-time variants that respectively bound
the transient behavior of $\dot x = Ax$ and $x_{k+1} = Ax_k$.
More specifically, the discrete-time version of the Kreiss Matrix Theorem \cite{Kre62}, after being refined by many authors over nearly thirty years, states that \cite[Theorem~18.1]{TreE05}
\beq
	\label{eq:kreiss_thm_disc}
	\Kcon \leq \sup_{k \geq 0} \|A^k\| \leq \e\n\Kcon,
\eeq
where the \emph{Kreiss constant} $\Kcon$ has 
two equivalent definitions \cite[p.~143]{TreE05}
\bseq	
	\label{eq:kcon_disc}
	\begin{alignat}{3}	
	\label{eq:k2d_disc}
	\Kcon 
	=& \sup_{z \in \C, |z| > 1} && \, (|z|  - 1) \| (z I - A)^{-1} \|, \\
	\label{eq:k1d_disc}
	=& \quad \sup_{\eps > 0} && \, \frac{\rhoeps(A) - 1}{\eps},
	\end{alignat}
\eseq
and the \emph{$\eps$-pseudospectral radius} $\rhoeps(\cdot)$ is defined by
\bseq
\begin{align}
	\rhoeps(A) 	&= \max \{ |z|  : z \in \Lambda(A+\Delta), \|\Delta\| \leq \eps \}, \\
				&= \max \{ |z|  : z \in \C, \| (zI - A)^{-1} \| \geq \eps^{-1} \}.
\end{align}
\eseq
For $\eps =0$, $\rhoeps(A) = \rho(A)$, the \emph{spectral radius} of $A$, and so it is easy
to see that $\Kcon = \infty$ if $\rho(A) > 1$.
Furthermore, if $A$ is normal and $\rho(A) \leq 1$, then $\Kcon = 1$, 
which is the minimum value $\Kcon$ can take.

The continuous-time Kreiss Matrix Theorem states that
\cite[Theorem~18.5]{TreE05}
\beq
	\label{eq:kreiss_thm_cont}
	\Kcon \leq \sup_{t \geq 0} \|\e^{tA}\| \leq \e\n\Kcon,
\eeq
where this version of $\Kcon$ also has two equivalent definitions \cite[eq.~14.7]{TreE05}
\bseq	
	\label{eq:kcon_cont}
	\begin{alignat}{3}
	\label{eq:k2d_cont}
	\Kcon 
	= & \sup_{z \in \C, \Re z > 0} && \, (\Re z) \| (z I - A)^{-1} \|, \\
	\label{eq:k1d_cont}
	= & \quad \ \sup_{\eps > 0} && \, \frac{\aleps(A)}{\eps},
	\end{alignat}
\eseq
and the \emph{$\eps$-pseudospectral abscissa} $\aleps(\cdot)$ is defined by
\bseq
\begin{align}
	\aleps(A) 	&= \max \{ \Re z  : z \in \Lambda(A+\Delta), \|\Delta\| \leq \eps \}, \\
				&= \max \{ \Re z  : z \in \C, \| (zI - A)^{-1} \| \geq \eps^{-1} \}.
\end{align}
\eseq
If $\eps =0$, $\aleps(A) = \alpha(A)$, the \emph{spectral abscissa} of $A$, and so $\Kcon = \infty$ if $\alpha(A) > 0$.  
Similar to the discrete-time case, $\Kcon \geq 1$ always holds
and $\Kcon = 1$  if $A$ is normal and $\alpha(A) \leq 0$.

In \cite{Mit19}, we introduced the first globally convergent algorithms to compute 
continuous- and discrete-time Kreiss constants to arbitrary accuracy.
Prior to this, it was only possible to estimate $\Kcon$ using supervised techniques,
i.e., where a user is an active participant of the process.
In \cite[Chapter~3.4.1]{Men06} and \cite{EmbK17}, 
Kreiss constants were approximated by 
plotting \eqref{eq:k1d_disc} or \eqref{eq:k1d_cont} 
and simply taking the maximum of the resulting curve.
Meanwhile, Kreiss constant estimation via
plotting $\| \e^{tA}\|$ with respect to $t$ or $\|A^k\|$ with respect to $k$
or by finding local maximizers of \eqref{eq:k1d_disc} or \eqref{eq:k1d_cont}
via optimization is discussed in \cite[Chapters~14 and 15]{TreE05}.
Plotting and/or grid techniques of course have low fidelity.  
They are unlikely to obtain the value of $\Kcon$
to more than a few digits at best
and may require a large number of function evaluations to have any accuracy whatsoever.
In contrast, under sufficient regularity conditions,
optimization techniques 
have high fidelity in finding local maximizers, often with relatively few function evaluations.
However, as the optimization problems in \eqref{eq:kcon_disc} and \eqref{eq:kcon_cont} are typically nonconvex,
general optimization solvers cannot guarantee that a global maximizer is found, and estimates from local minimizers
can be arbitrarily bad approximations to $\Kcon$.
Even if one happens to know a relatively small bounded region containing a global maximizer of the optimization problems in 
\eqref{eq:kcon_disc} and \eqref{eq:kcon_cont},
to guarantee any level of accuracy, 
this region must still be sufficiently sampled (for plotting or grid techniques) or contain no other stationary points (for optimization).
Knowing such a region \emph{and} how much sampling is required or that it contains no other stationary points is not typical, at least not without user experimentation.
Moreover, if transient behavior occurs on a fast time scale, such regions can be very small and thus hard to find, particularly without fine-grained sampling.
As  noted by Mengi \cite[Section~6.2.2]{Men06}, ``in general it is difficult to 
guess \emph{a priori} which $\eps$ value is most relevant for the transient peak [of \eqref{eq:k1d_disc} or \eqref{eq:k1d_cont}]." 

For our recent algorithms to compute Kreiss constants with theoretical guarantees \cite{Mit19},
we actually worked with the inverses of \eqref{eq:k2d_disc} and \eqref{eq:k2d_cont}, respectively
\bseq
	\label{eq:kinv_eqs}
	\begin{alignat}{3} 
	\label{eq:kinv_disc}
	\qquad \Kinv =& \ \: \inf_{| z | > 1} 
	&& \: \smin \left( \frac{z I - A}{|z| - 1} \right) \qquad \text{(discrete-time)}, \\
	\label{eq:kinv_cont}
	\qquad \Kinv =& \ \inf_{\Re z > 0} 
	&& \: \smin \left( \frac{z I - A}{\Re z} \right) \qquad \text{(continuous-time)}.
	\end{alignat}
\eseq
In this form, it is easier to see that the optimization problems in \eqref{eq:kinv_eqs} have some similarity to
\eqref{eq:dtu}, which naturally leads to the question of whether or not 
any of the aforementioned $\dtu$ methods could be adapted to computing $\Kcon$.
Like the $\dtu$ setting, the objective functions in \eqref{eq:kinv_eqs} are semi-algebriac, so they 
have a finite number of locally minimal values; hence, properly designed optimization-with-restart
algorithms will converge to $\Kinv$ within a finite number of restarts.
Optimization can also robustly and efficiently find (feasible) minimizers of \eqref{eq:kinv_eqs} in order to obtain locally optimal approximations to $\Kinv$, even for large scale problems.
For complete details on both of these points, see our previous comments in \S\ref{sec:intro_dtu} and \cite[Section~3]{Mit19}.
However, as shown in \cite[Section~4]{Mit19}, there are in fact fundamental differences
between computing the distance to uncontrollability and Kreiss constants
and existing $\dtu$ methods do not extend directly.
Nevertheless, for the objective functions in \eqref{eq:kinv_eqs},
we developed several different Kreiss constant analogues of Gu's \cite[Theorem~3.1]{Gu00}, which 
along with several new 2D level-set tests, enable three different $\Kcon$ iterations \cite{Mit19}.

When computing continuous-time $\Kcon$, the first of these is based on a new 2D~level-set test
that, similar to Gu's $\dtu$ test for \eqref{eq:gu_test}, looks for pairs of level-set points of the objective 
function in \eqref{eq:kinv_cont} that are a \emph{fixed-distance} $\eta$ apart.  
However, the aforementioned $\dtu$ trisection and optimization-with-restart algorithms of \cite{BurLO04}
cannot be used with this new test because
a meaningful lower bound for $\Kinv$, like \eqref{eq:dtu_lb}, is \emph{not} asserted
 when no such level-set pairs are detected; see \cite[Section~4]{Mit19}.
But, by combining optimization-with-restarts with a \emph{backtracking procedure}, it 
was possible to use this fixed-distance test to devise a globally convergent iteration for $\Kcon$;
see \cite[Section~5]{Mit19}.
This new level-set involves computing all real eigenvalues of 
a generalized eigenvalue problem $\mathcal{A}_1 - \lambda\mathcal{A}_2$ of order $4n^2$ \cite[eq.~(5.5)]{Mit19}, 
where $\mathcal{A}_2$ is singular with rank $2n^2$.
As noted in \cite[Section~5.3]{Mit19}, it does not seem possible to analytically reduce 
the order of $\mathcal{A}_1 - \lambda\mathcal{A}_2$ to $2n^2$ or to a standard eigenvalue problem
via the techniques of \cite{GuMOetal06} for Gu's $\dtu$ level-set test.

Meanwhile, in \cite[Section~6]{Mit19}, we devised a second 2D level-set test for \eqref{eq:kinv_cont} 
that looks for pairs of level-set points that are a certain \emph{variable-distance} apart involving $\eta$.  This too leads to solving 
a generalized eigenvalue problem of order $4n^2$, \mbox{$\mathcal{B}_1 - \lambda\mathcal{B}_2$} \cite[eq.~(6.6)]{Mit19}, 
but here $\mathcal{B}_2$ is nonsingular and we derived an explicit form for its inverse.
Differences in numerical reliability between the fixed- and variable-distance tests 
are still not entirely clear, but one advantage of 
the variable-distance test is that it \emph{does} assert that 
the lower bound $\Kinv > \gamma - \tfrac{\eta}{2}$ must hold whenever no such level-set pairs are detected.
Thus, this variable-distance 2D level-set test enables two more $\Kcon$ iterations, which respectively use trisection and optimization-with-restarts \emph{without backtracking}.
In practice, optimization-with-restarts without backtracking generally needs the least number of 
level-set tests, while trisection needs far more than either of the optimization-based iterations.

To compute discrete-time $\Kcon$, we also developed three analogues of these iterations \cite[Section~7]{Mit19},
which, at a very high level, work similarly.
However,
the underlying new 2D level-set tests for the objective function in \eqref{eq:kinv_disc} are substantially 
different and more complicated and expensive.  In particular, they require solving \emph{quadratic} eigenvalue problems of order $4n^2$.

Like the $\dtu$ methods, these $\Kcon$ methods do $\bigO(n^6)$ work when using dense eigensolvers.
However, solving the larger and generalized/quadratic eigenvalue problems
is substantially more expensive and requires much more memory,
particularly for discrete-time $\Kcon$.
In \cite[Sections~5.4, 6.3, and 7.4]{Mit19},
we also developed divide-and-conquer versions of all of these $\Kcon$ algorithms.
While these variants do $\bigO(n^4)$ work on average and $\bigO(n^5)$ in the worst case and have dramatically
lower memory requirements, we noted in \cite[Section~8]{Mit19}
that divide-and-conquer for the $\Kcon$ setting does not appear to be very reliable in practice.
All our 2D level-set tests for $\Kcon$ also use our improved 
technique that is explained in \cite[Key~Remark~6.3]{Mit19}.

\subsection{Motivation and contribution of the paper}
As we have just seen, the state-of-the-art methods for computing $\dtu$ and $\Kcon$
are based on a 2D level-set test methodology that intrinsically involves solving very large eigenvalue problems.
Even in their faster divide-and-conquer variants, these algorithms
are prohibitively expensive.
Moreover, 
the convergence guarantees of these methods assume exact computation,
but rounding errors in computed eigenvalues may cause the methods to fail numerically.
Our aforementioned modified technique to perform these 2D level-set tests more reliably, 
while effective and in fact crucial for robust codes, 
does not address all the numerical pitfalls of these methods.

In this paper, we address both these high-cost and reliability issues
by proposing a new methodology for computing quantities
whose values are given by global optimization of singular value functions in two real variables.
We do this by developing new level-set tests for optimization-with-restarts-based methods, 
which we call \emph{interpolation-based globality certificates} and 
that work by sufficiently resolving certain one-variable continuous functions 
over a \emph{finite interval known a priori}.
These new functions are reasonably well behaved and relatively cheap and robust to evaluate, 
all of which makes high fidelity approximation via interpolation practical.
Our new $\dtu$ and $\Kcon$ methods have $\bigO(kn^3)$ work complexity and require $\bigO(n^2)$ memory,
where $k$, the total number of function evaluations incurred to build the interpolants,
is such that our new methods are orders of magnitude faster than the previous
state-of-the-art.  
Moreover, additional significant speedups can be attained via parallel processing, 
since function evaluations for interpolation are ``embarrassingly parallel."
Our ``strength in numbers" interpolation-based approach also has numerical benefits,
as global convergence does not crucially hinge upon any \emph{single} computation 
being susceptible to rounding error, which is not true for almost all of the 2D level-set test methods discussed in \S\ref{sec:intro_dtu} and  \S\ref{sec:intro_kreiss} (the sole exception is the Kreiss constant iteration
using backtracking).
The trade-off we have made here is that instead of putting our faith 
in accurately computing eigenvalues of very large eigenvalue problems,
we assume that approximation via interpolation is reliable enough to be used as a subroutine.
In this sense, our new approach can be considered complementary to our earlier efforts of \cite{Mit19},
as they are built on very different foundations.

In \S\ref{sec:kcont} we present our new interpolation-based globality certificates 
for the case of computing continuous-time $\Kcon$.
Analogues of our interpolation-based certificates for discrete-time $\Kcon$ and $\dtu$
are respectively derived in \S\ref{sec:kdisc} and \S\ref{sec:dtu}.
Numerical experiments are given in \S\ref{sec:experiments}, 
while concluding remarks are made in \S\ref{sec:conclusion}.

\section{A new approach for computing continuous-time Kreiss constants}
\label{sec:kcont}
We now propose a new optimization-with-restarts-based algorithm for computing continuous-time $\Kcon$,
i.e., a new method to find a global minimizer of \eqref{eq:kinv_cont}.
As previously mentioned, local minimizers of \eqref{eq:kinv_cont} can be found 
relatively cheaply, and its objective function has a finite number of locally minimal values.
Thus, given a corresponding level-set test,
only a finite number of optimization restarts are necessary to compute $\Kinv$, and equivalently $\Kcon$,
to arbitrary accuracy.
However, unlike the earlier methods discussed in the introduction, 
we abandon the concept of looking for \emph{pairs of points on the $\gamma$-level set} of
the given singular value function for which a global minimizer is sought.
Instead, we focus on devising a new type of level-set test, which given some $\gamma \geq \Kinv$
corresponding to a minimizer of \eqref{eq:kinv_cont}, 
answers the question: are there other points on this same level set, and if so, where are they?  

For \eqref{eq:kinv_cont}, minimizers should be computed using Cartesian coordinates 
(see \cite[Section~3.1]{Mit19} for full details), but note that 
all of our \emph{interpolation-based globality certificates} to detect level-set points
are based on polar coordinates, even in continuous-time settings.
Thus, consider \eqref{eq:kinv_cont} parameterized in polar coordinates:
\beq
	\label{eq:g_fns}
	g(r,\theta) \coloneqq \smin(G(r,\theta)) 
	\quad \text{and} \quad 
	G(r,\theta) \coloneqq \frac{r\eit I - A}{r\costh},
\eeq
so 
\[ 
	\Kinv = \inf_{r > 0, \, \theta \in (-\tfrac{\pi}{2},\tfrac{\pi}{2})} \, g(r,\theta).
\]
As computing $\Kcon$ is trivial if either $A$ is normal or $\alpha(A) > 0$, 
we assume neither holds.
For reasons that will become clear momentarily,
we also assume that $0 \not\in \Lambda(A)$.

\subsection{Level sets of $g(r,\theta)$ and a 1D radial level-set test}
\label{sec:level_sets_kc}
For a fixed $\theta \in \R$, 
the following key result relates singular values of $G(r,\theta)$ with eigenvalues
of a certain $2\n \times 2\n$ matrix pencil.
Exploiting such relationships of singular values and eigenvalues has a rich history 
in computing various robust stability measures, starting when 
Byers introduced the first algorithm to compute the distance to instability in 1988 \cite{Bye88}.

\begin{theorem}
\label{thm:MN_kc}
Let $A \in \C^{n\times n}$ and $\gamma, r,\theta \in \R$ with $r \neq 0$.
Then $\gamma \geq 0$ is a singular value of $G(r,\theta)$ defined in \eqref{eq:g_fns} 
if and only if $\imagunit r$ is an eigenvalue of the skew-Hamiltonian-Hamiltonian 
matrix pencil $(M,N_\theta)$, where
\beq
	\label{eq:eigMN_kc}
	M \coloneqq
	\begin{bmatrix}
		A  	& 0 \\
		0 	& -A^*
	\end{bmatrix}
	\quad \text{and} \quad
	N_\theta \coloneqq
	\begin{bmatrix}
		-\imagunit \eit I 				& \imagunit \gamma \costh I  \\
	 	- \imagunit \gamma \costh I    	& \imagunit \emit I
	\end{bmatrix},
\eeq
$N_\theta$ is singular if and only if $| \gamma \costh | = 1$,
and $(M,N_\theta)$ is regular if $|\gamma \costh| \neq 1$.
\end{theorem}
\begin{proof}
It is easy to verify that $M$ and $N_\theta$ are respectively Hamiltonian and skew-Hamiltonian,
and so $(M,N_\theta)$ is an sHH matrix pencil.
As $N_\theta$ is composed of four blocks of different multiples of the $\n \times \n$ identity matrix, 
$\det(N_\theta) = 1 - (\gamma \costh)^2$.
Thus, $| \gamma \costh| \neq 1$, i.e., $N_\theta$ being nonsingular, is clearly a
sufficient condition for $(M,N_\theta)$ to be a regular matrix pencil.
Now suppose $\gamma$ is a singular value of $G(r,\theta)$ with left and right singular vectors $u$ and $v$.
Then the following two equations hold:
\[
	\gamma
	\begin{bmatrix}
		u \\ v
	\end{bmatrix}
	=
	\begin{bmatrix}
		G(r,\theta)  	& 0 \\
		0 			& G(r,\theta)^*
	\end{bmatrix}
	\begin{bmatrix}
		v \\ u
	\end{bmatrix}
	\ \ \text{and} \ \
	\gamma r \costh
	\begin{bmatrix}
		u \\ v
	\end{bmatrix}
	=
	\begin{bmatrix}
		r\eit I - A  	& 0 \\
		0 		& r\emit I - A^*
	\end{bmatrix}
	\begin{bmatrix}
		v \\ u
	\end{bmatrix}.
\]
Multiplying the bottom block by -1 and rearranging terms, this is equivalent to
\[
	\begin{bmatrix}
		A  	& 0 \\
		0 	& -A^*
	\end{bmatrix}
	\begin{bmatrix}
		v \\ u
	\end{bmatrix}
	=
	r
	\begin{bmatrix}
		\eit I  			& -\gamma \costh I \\
		\gamma \costh I 	& -\emit I
	\end{bmatrix}
	\begin{bmatrix}
		v \\ u
	\end{bmatrix}.
\]
Noting that the matrix on the right multiplied by $-\imagunit$ is $N_\theta$ 
completes the proof.
\end{proof}

\begin{remark}
\label{rem:min_sv}
By \cref{thm:MN_kc}, if a point $(\tilde r,\tilde \theta)$ is in the $\gamma$-level set of $g(r,\theta)$ for some $\gamma \geq 0$,
then $\imagunit \tilde r \in \Lambda(M,N_{\tilde \theta})$.
Note that the converse is not necessarily true.
If $\imagunit \tilde r$ is an eigenvalue of the matrix pencil $(M,N_{\tilde \theta})$, \cref{thm:MN_kc} only states that 
$\gamma$ is a singular value of $G(\tilde r,\tilde \theta)$.  
For point $(\tilde r,\tilde \theta)$ to be in the $\gamma$-level set, $\gamma$ would additionally have to be the 
smallest singular value of $G(\tilde r,\tilde \theta)$.
However, if $\gamma$ is not the minimum singular value of $G(r,\theta)$, 
then $(\tilde r,\tilde \theta)$ is instead in some $\hat \gamma$-level set of $g(r,\theta)$ with $\hat \gamma < \gamma$.
\end{remark}

Besides being computationally useful for detecting level-set points, 
\cref{thm:MN_kc} provides a way to show that the 
$\gamma$-level set of $g(r,\theta)$ is bounded for $\gamma \in [0,1)$.

\begin{theorem}
\label{thm:bounded_cont}
Let $A \in \C^{n\times n}$ and $\gamma \in [0,1)$.
The $\gamma$-level set of $g(r,\theta)$ defined in \eqref{eq:g_fns} is bounded.
Moreover, if $\alpha(A) < 0$, the $\gamma$-level set is compact.
\end{theorem}
\begin{proof}
For any point $\tilde r \eit$ in the $\gamma$-level set of $g(r,\theta)$, $\gamma$ is a singular value of $G(\tilde r,\theta)$.
Thus by \cref{thm:MN_kc}, $\imagunit \tilde r \in \Lambda(M,N_\theta)$.
Furthermore, all eigenvalues of $(M,N_\theta)$ must be finite, as $|\gamma| < 1$ implies that $N_\theta$ is always nonsingular.
Consider the function $m(\theta) \coloneqq \rho(M,N_\theta)$.
By continuity of the spectral radius,
$m(\theta)$ must have a finite maximal value $m_\star$
on $[-\tfrac{\pi}{2},\tfrac{\pi}{2}]$. Since $|\imagunit \tilde r | \leq m_\star$ must hold,
the $\gamma$-level set of $g(r,\theta)$ is bounded.
If $\alpha(A) < 0$, then $g(r,\theta)$ is infinite on all of the imaginary axis, and so its $\gamma$-level set 
cannot contain purely imaginary values.  Thus,
the $\gamma$-level set must additionally be in the open right half-plane and closed, hence compact.
\end{proof}

Our new method will require that zero is not an eigenvalue of $(M,N_\theta)$.  
The following theorem gives the precise conditions to meet this requirement, namely that
zero cannot be an eigenvalue of $A$.  

\begin{theorem}
\label{thm:zero_kc}
Let $A \in \C^{n\times n}$ and $\gamma,\theta \in \R$.
Then the matrix pencil $(M,N_\theta)$ defined by \eqref{eq:eigMN_kc} has zero as an eigenvalue 
if and only if the matrix $A$ also has zero as an eigenvalue.
Consequently, $0 \not\in \Lambda(A)$ also ensures that $(M,N_\theta)$ is regular.
\end{theorem}
\begin{proof}
The proof is immediate as $\det(M) = \det(A)\det(-A^*)$, 
so clearly $(M,N_\theta)$ is a regular matrix pencil if $0 \not \in \Lambda(A)$.
\end{proof}

Given $\gamma \geq \Kinv$ and some $\theta \in (-\tfrac{\pi}{2},\tfrac{\pi}{2})$, 
\cref{thm:MN_kc} provides a way to compute all the $\gamma$-level set points of $g(r,\theta)$
along the ray emanating from the origin specified by angle $\tilde \theta$,
namely via computing all the imaginary eigenvalues of $(M,N_{\tilde \theta})$.
As $(M,N_\theta)$ is an sHH pencil, its eigenvalues are symmetric with respect to the imaginary axis,
and structure-preserving eigenvalues solvers such as \cite{BenBMetal02} can be used to ensure
that computed imaginary eigenvalues are exactly on the imaginary axis in the presence of rounding errors.
While this is desirable for numerical robustness, solving generalized eigenvalue problems
is many times more expensive than solving a standard eigenvalue problem of the same dimension
and, as we explain later, our new methods often do not need this level of robustness.
Thus, since $N_\theta$ is generically nonsingular, 
we now investigate the condition number of $N_\theta$ to ascertain the feasibility
of instead computing the eigenvalues of $N_\theta^{-1}M$ via the QR algorithm.
We first need the following generic result.

\begin{lemma}
\label{lem:condnum}
Let the matrix $E = \begin{bsmallmatrix} aI & bI \\ \overline{b}I & \overline{a}I \end{bsmallmatrix}$
with $a,b \in \C$ such that $\det(E) \neq 0$.
Then the condition number of $E$ is
\[
	\kappa(E) = \tfrac{|a| + |b|}{| |a| - |b| |}.
\]
\end{lemma}
\begin{proof}
Since $E$ is nonsingular, all its singular values are positive and they are equal to the 
square roots of the eigenvalues of $EE^*$.  As
\[
	EE^* = 
	\begin{bsmallmatrix} 
		(a\overline{a} + b\overline{b}) I  & 2ab I \\
		2 \overline{ab} I 			& (a\overline{a} + b\overline{b}) I
	\end{bsmallmatrix}
	\eqqcolon
	\begin{bsmallmatrix} 
		c I  			& d I \\
		\overline{d} I 	& c I
	\end{bsmallmatrix},
\]
$0= \det(EE^* - \lambda I) = \lambda^2 - 2c\lambda + (c^2 - |d|^2)$,
so the eigenvalues of $EE^*$ are $\lambda = c \pm |d| = |a|^2 + |b|^2 \pm 2|ab|$.
Thus, the singular values of $E$ are $| |a| \pm |b| |$.
\end{proof}

\begin{theorem}
\label{thm:simpleMN_kc}
Let $A \in \C^{n \times n}$ and $\gamma,\theta \in \R$ with $| \gamma \costh | \neq 1$. 
Then the spectrum of the matrix pencil $(M,N_\theta)$ defined by \eqref{eq:eigMN_kc} is 
equal to the spectrum of 
\beq
	\label{eq:eigM_kc}
	M_\theta \coloneqq N_\theta^{-1} M = 
	\frac{\imagunit}{1 - (\gamma \costh)^2}
	\begin{bmatrix}
	\emit A 			& (\gamma \costh) A^*  \\
	(\gamma \costh) A  	& \eit A^*
	\end{bmatrix},
\eeq
and if $\gamma \in [0,1)$, then $\max_{\theta \in \R} \kappa(N_\theta) = \tfrac{1+\gamma}{1 - \gamma}$.
\end{theorem}
\begin{proof}
The matrix given in \eqref{eq:eigM_kc} simply follows by using the obvious explicit form of $N_\theta^{-1}$,
which exists if and only if $| \gamma \costh | \neq 1$, and then evaluating $N_\theta^{-1}M$.
Applying \cref{lem:condnum} to $N_\theta$, we have that 
$\kappa(N_\theta) = \tfrac{1 + |\gamma \costh|}{|1 - |\gamma \costh||}$.
It is easy to see that if $\gamma \in [0,1)$, then $\theta = 0$ is a global maximizer of this ratio, thus completing the proof.
\end{proof}

For $\gamma = 0.9$, \cref{thm:simpleMN_kc} says that the condition number of $N_\theta$ is only 19, and $\kappa(N_\theta) \to 1$ monotonically as $\gamma \to 0$.      
While $\kappa(N_\theta)$ does blow up as $\gamma \to 1$, this is mostly inconsequential, 
since $\gamma=\Kinv \in [0.9,1]$ corresponds to Kreiss constants between 1 and 1.1.  
In other words, for almost all matrices of interest, encountered values of $\gamma$ should be much less 
than 0.9, and so $N_\theta$ will be very well conditioned.
Hence, there is generally no numerical concern in computing the eigenvalues of $(M,N_\theta)$ via $N_\theta^{-1}M$,
except that the imaginary axis symmetry will not be maintained exactly via the standard QR algorithm.
As we clarify later, computing the spectrum of $(M,N_\theta)$ using an sHH structure-preserving
eigensolver can always be done as a backup.

\subsection{An interpolation-based globality certificate for $g(r,\theta)$}
We are now ready to present our first interpolation-based globality certificate, specifically for \eqref{eq:kinv_cont}.  
Given $\gamma \geq 0$, 
the idea is to sweep the open right half of the complex plane with rays from the origin to determine
which ones intersect the $\gamma$-level set.
To do this, we are about to construct a rather well behaved continuous function $\sgfn: (-\tfrac{\pi}{2},\tfrac{\pi}{2}) \mapsto [0,\pi^2]$
such that $\sgfn(\tilde \theta) = 0$ holds whenever the ray from the origin determined by angle $\tilde \theta$ intersects the $\gamma$-level set of $g(r,\theta)$.
Hence, if $\sgfn(\theta)$ is strictly positive for all $\theta \in (-\tfrac{\pi}{2},\tfrac{\pi}{2})$,
then $\gamma < \Kinv$ must hold.
Otherwise, the angles $\tilde \theta$ for which $\sgfn(\tilde \theta) = 0$ provide the directions of the rays 
that intersect the $\gamma$-level, and provided these intersection points are not stationary, they 
can be used to restart optimization to find better (lower) minimizers of \eqref{eq:kinv_cont}.
By approximating $\sgfn(\theta)$ via interpolation, we can then ascertain if it has any zeros.

Keeping in mind that the spectrum of $(M,N_\theta)$ as defined by \eqref{eq:eigMN_kc} is always imaginary-axis symmetric, 
to accomplish our criteria above, consider
\bseq
	\label{eq:gamma_kc}
	\begin{align}
	\label{eq:dist_kc}
	\sgfn(\theta) &\coloneqq
	\min \{ \Arg(-\imagunit \lambda)^2 : \lambda \in \Lambda(M,N_\theta), \Re \lambda \leq 0 \}, \\					
	\label{eq:set_kc}
	\sgset(\gamma) &\coloneqq
	\interior \{ \theta : \sgfn(\theta) = 0, \, \theta \in (-\tfrac{\pi}{2},\tfrac{\pi}{2}) \},	
	\end{align}
\eseq
where $\Arg : \C \setminus \{0\} \mapsto (-\pi,\pi]$ is the principal value argument function.

\begin{theorem}
\label{thm:props_kc}
Let $A \in \C^{n \times n}$ with $\alpha(A) \leq 0$ and $0 \not\in \Lambda(A)$.
Then for any $\gamma \geq 0$, the function $\sgfn(\theta)$ defined in \eqref{eq:dist_kc} has the following properties:
\begin{enumerate}[label=\roman*.]
\item $\sgfn(\theta) \geq 0$ for all $\theta \in \mathcal{D} \coloneqq (-\tfrac{\pi}{2},\tfrac{\pi}{2})$,
\item $\sgfn(\theta) = 0$ if and only if there exists $\imagunit r \in \Lambda(M,N_\theta)$ with $r \in \R$ and $r > 0$,
\item $\sgfn(\theta)$ is continuous on its entire domain $\mathcal{D}$, 
\item $\sgfn(\theta)$ is differentiable at a point $\theta$ if the eigenvalue $\lambda \in \Lambda(M,N_\theta)$ attaining the value of $\sgfn(\theta)$ is unique and simple.
\end{enumerate}
Furthermore,  the following properties hold for the set $\sgset(\gamma)$ defined in \eqref{eq:set_kc}:
\begin{enumerate}[label=\roman*.,resume]
\item if $\Kinv < \gamma$, then 
	$0 < \mu(\sgset(\gamma))$,
\item $\gamma_1 \leq \gamma_2$ if and only if $\mu (\sgset(\gamma_1)) \leq \mu (\sgset(\gamma_2))$,
\item $\lim_{\gamma \to \infty} \mu(\sgset(\gamma)) = \pi$,
\end{enumerate}
where $\mu(\cdot)$ is the Lebesgue measure on $\R$.
\end{theorem}
\begin{proof}
We begin with $\sgfn(\theta)$.
The first and second properties hold by construction, since $-\imagunit \lambda$ in \eqref{eq:dist_kc} is always in the (closed)
upper half of the complex plane.  The third property is a consequence
of the continuity of eigenvalues and our assumption that $0 \not\in \Lambda(A)$,
which via \cref{thm:zero_kc}, ensures that zero is never an eigenvalue of $(M,N_\theta)$ for any $\theta$.
The fourth property follows from standard perturbation theory for simple eigenvalues and by the definition of $\sgfn(\theta)$.

We now turn to $\sgset(\gamma)$, which since it is defined as an interior, is thus open and measurable.
Let $(r_\star,\theta_\star)$ be a global minimizer, i.e., $g(r_\star,\theta_\star) = \Kinv$,
where $r_\star > 0$ and $\theta_\star \in \mathcal{D}$, and 
$\mathcal{L}_\gamma \coloneqq \{ (r,\theta) : g(r,\theta) < \gamma, r > 0, \theta \in \mathcal{D})\}$
be a strict lower level set of $g(r,\theta)$.
As $\mathcal{L}_\gamma$ is open, there exists an open disk neighborhood $\mathcal{N} \subset \mathcal{L}_\gamma$ about $(r_\star,\theta_\star)$,
so let $\mathcal{T}$ denote the (positive-length) interval of angles specifying the rays from the origin that intersect $\mathcal{N}$.
Since $g(r,\theta) < \gamma$ for all points in $\mathcal{N}$ and $\lim_{r \to 0^+} g(r,\theta) = \infty$ for any $\theta \in \mathcal{D}$ (as $0 \not\in \Lambda(A)$), it follows by continuity of $g(r,\theta)$ that for every $\theta \in \mathcal{T}$ there exists at least one $\tilde r \in (0,r_\star)$
such that $g(\tilde r,\theta) = \gamma$.  Thus, by \cref{thm:MN_kc} it follows that $\sgfn(\theta) = 0$ for all $\theta \in \mathcal{T}$, 
and as $\mu(\mathcal{T}) > 0$ and $\mathcal{T} \subset \sgset(\gamma)$, the fifth property holds.
The sixth property holds by noting that $\gamma_1 \leq \gamma_2$ if and only if $\mathcal{L}_{\gamma_1} \subseteq \mathcal{L}_{\gamma_2}$,
which in turn is equivalent to $\sgset(\gamma_1) \subseteq \sgset(\gamma_2)$. 
To see this, consider any ray from the origin that intersects $\bd \mathcal{L}_{\gamma_1}$, say, at point $(\hat r,\theta)$.  Then $g(\hat r,\theta) = \gamma_1$,
and so as in the argument for the fifth property, there exists $\tilde r \in (0,\hat r)$ such that $g(\tilde r,\theta) = \gamma_2$;
hence this ray must intersect $\bd \mathcal{L}_{\gamma_2}$ at $(\tilde r,\theta)$.  Thus, $g_{\gamma_1}(\theta) = 0$ implies $g_{\gamma_2}(\theta) = 0$,
and so $\sgset(\gamma_1) \subseteq \sgset(\gamma_2)$.
Now suppose $\sgset(\gamma_1) \supset \sgset(\gamma_2)$ and let $\theta \in \sgset(\gamma_1) \setminus \sgset(\gamma_2)$,
hence $g_{\gamma_1}(\theta) = 0$ but $g_{\gamma_2}(\theta) > 0$.
Then $g(\hat r,\theta) = \gamma_1$ holds for some $\hat r > 0$, but $g(r,\theta) \neq \gamma_2$ for all $r \in (0,\infty)$,
and so $\gamma_2 < \min_{r > 0} g(r,\theta) \leq \gamma_1$, a contradiction.
For the seventh property, we first note that $\lim_{r \to \infty} g(r,\theta) = \sec \theta \geq 1$; hence for any $\theta$
such that $\sec \theta \leq \gamma$, there again must exist  $\tilde r > 0$ such that $g(\tilde r,\theta) = \gamma$.
Thus, it is clear that $\lim_{\gamma \to \infty} \mu(\sgset(\gamma)) = \pi$ must hold.
\end{proof}

Taken together with Theorem~\ref{thm:MN_kc} and  Remark~\ref{rem:min_sv}, it clear that 
$\sgfn(\theta)$ meets our new criteria for a level-set test.
Given $\gamma \geq 0$, $\tilde r > 0$ and some $\tilde \theta \in \mathcal{D}$, 
if point $(\tilde r,\tilde \theta)$
is in the $\gamma$-level set of $g(r,\theta)$, then by Theorem~\ref{thm:MN_kc}, $\imagunit \tilde r$
must be an eigenvalue of matrix pencil $(M,N_{\tilde\theta})$ and so $\sgfn(\tilde \theta) = 0$ holds.
If $\sgfn(\tilde \theta) = 0$, 
by definition there exists $\imagunit \tilde r \in \Lambda(M,N_{\tilde\theta})$ with $r > 0$,
and so by Theorem~\ref{thm:MN_kc}, 
$\gamma$ must be a singular value of $G(\tilde r,\tilde \theta)$.
Thus by Remark~\ref{rem:min_sv}, point $(\tilde r,\tilde \theta)$ must either be in the $\gamma$-level set 
of $g(r,\theta)$ or some other $\tilde\gamma$-level set with $\tilde \gamma < \gamma$.
Hence, $\sgfn(\theta) = 0$ is associated with new starting points for optimization such that a better (lower) minimizer can be found.
Finally, if $\sgfn(\theta) > 0$ for all $\theta \in \mathcal{D}$,
then $(M,N_\theta)$ has no imaginary eigenvalues on the positive imaginary axis for any $\theta \in \mathcal{D}$, so again by Theorem~\ref{thm:MN_kc},
$\gamma$ is not a singular value of $G(r,\theta)$ 
for any $r > 0$ and $\theta \in \mathcal{D}$.
This in turn means the $\gamma$-level set of $g(r,\theta)$ is empty.
As $g(r,\theta)$ is continuous, $\gamma < \Kinv$ must hold.

\begin{remark}
As we will approximate $\sgfn(\theta)$ via interpolation,
the presence of the square in $\Arg(-\imagunit \lambda)^2$
is to help smooth out the numerically difficult high rate of change 
that $\Arg(-\imagunit \lambda)$ would otherwise have.
To understand this, suppose that the $\gamma$-level set of $g(r,\theta)$ 
consists of a single continuous closed curve enclosing a nonempty convex interior.
Then $\sgset(\gamma) \subset (-\tfrac{\pi}{2},\tfrac{\pi}{2})$ is simply a single interval 
and for any $\theta$ in $\sgset(\gamma)$, $(M,N_\theta)$ must have two
distinct eigenvalues $\imagunit r_1$ and $\imagunit r_2$ with $r_1,r_2 > 0$.
However, as~$\theta$ approaches either end of interval $\sgset(\gamma)$,
this pair will first coalesce on the imaginary axis and then split apart again,
with both eigenvalues moving very rapidly off of the imaginary axis (in opposite directions).
\end{remark}

\def\imgscale{0.33}

\begin{figure}[!t]
\centering
\subfloat[Level sets]{
\includegraphics[scale=\imgscale,trim={0.9cm 0.05cm 1.5cm 0.8cm},clip]{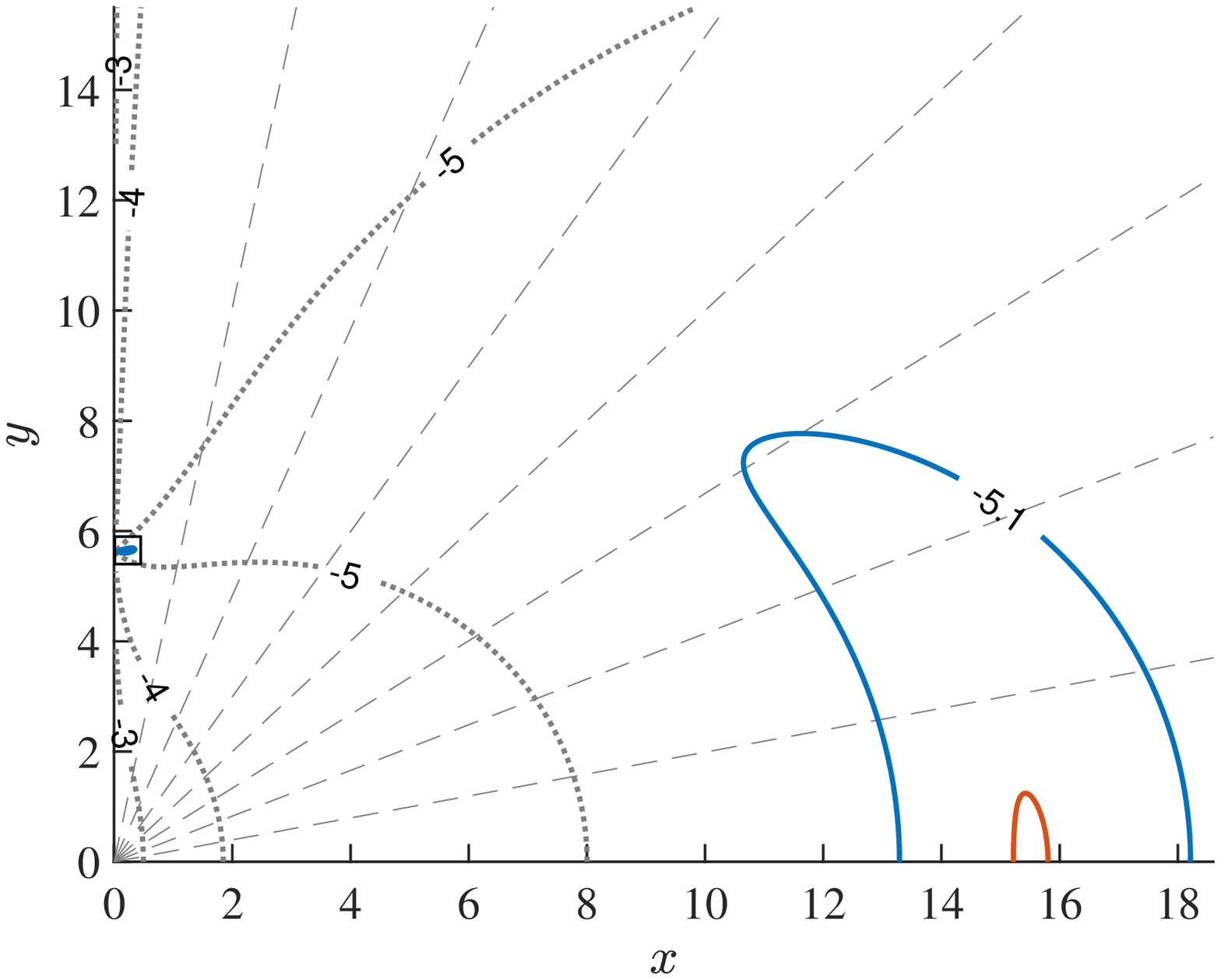} 
\label{fig:kcl}
} 
\subfloat[$\sgfn(\theta)$ on $\left[0,\tfrac{1}{2}\pi\right)$]{
\includegraphics[scale=\imgscale,trim={0.3cm 0.05cm 1.5cm 0.8cm},clip]{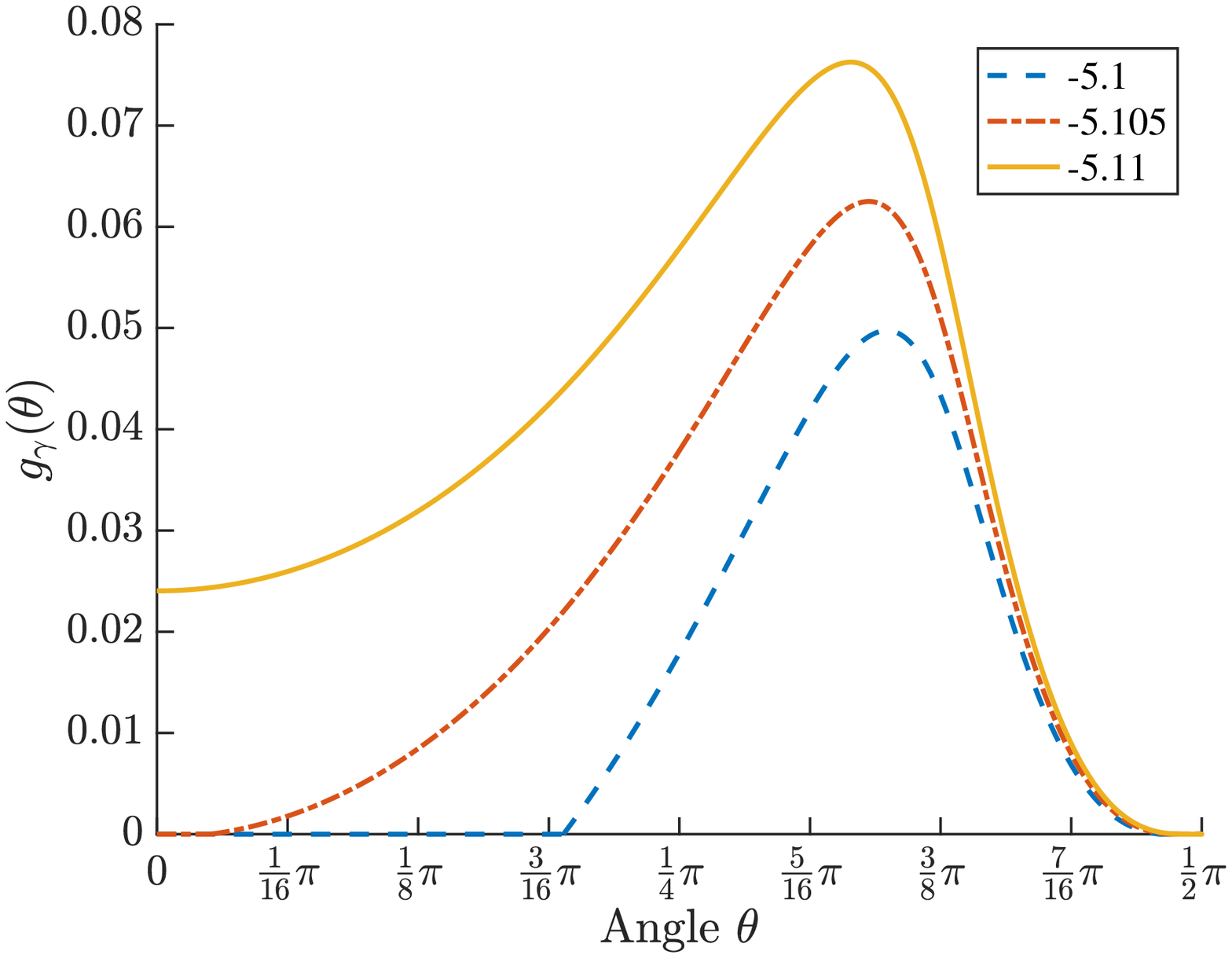} 
\label{fig:kcd}
}
\\
\subfloat[Level sets (enlarged view)]{
\includegraphics[scale=\imgscale,trim={0.9cm 0.05cm 1.3cm 0.3cm},clip]{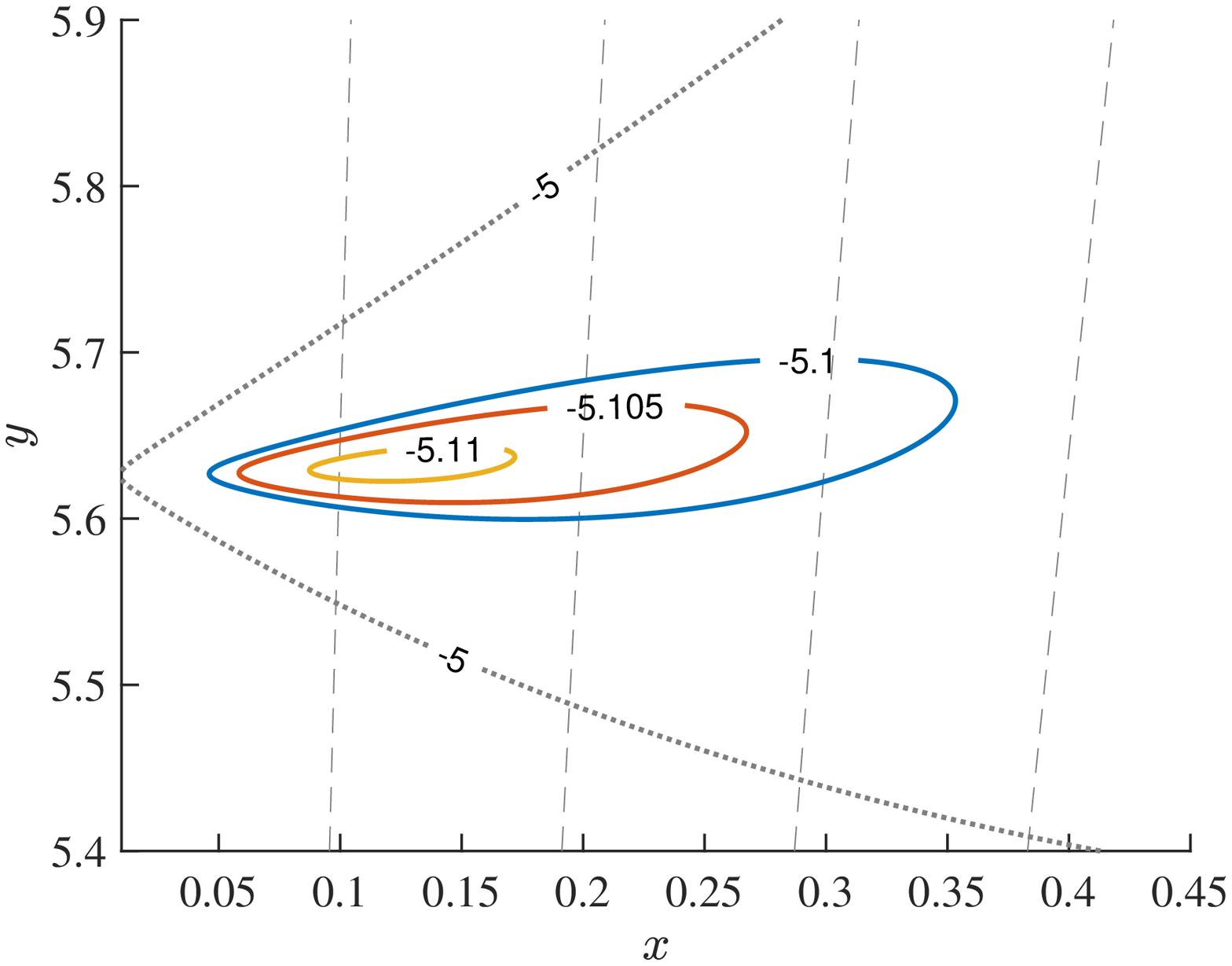} 
\label{fig:kcl_zoom}
}
\subfloat[$\sgfn(\theta)$ on $\left[\tfrac{3}{2},\tfrac{1}{2}\pi\right)$]{
\includegraphics[scale=\imgscale,trim={0.3cm 0.05cm 1.5cm 0.3cm},clip]{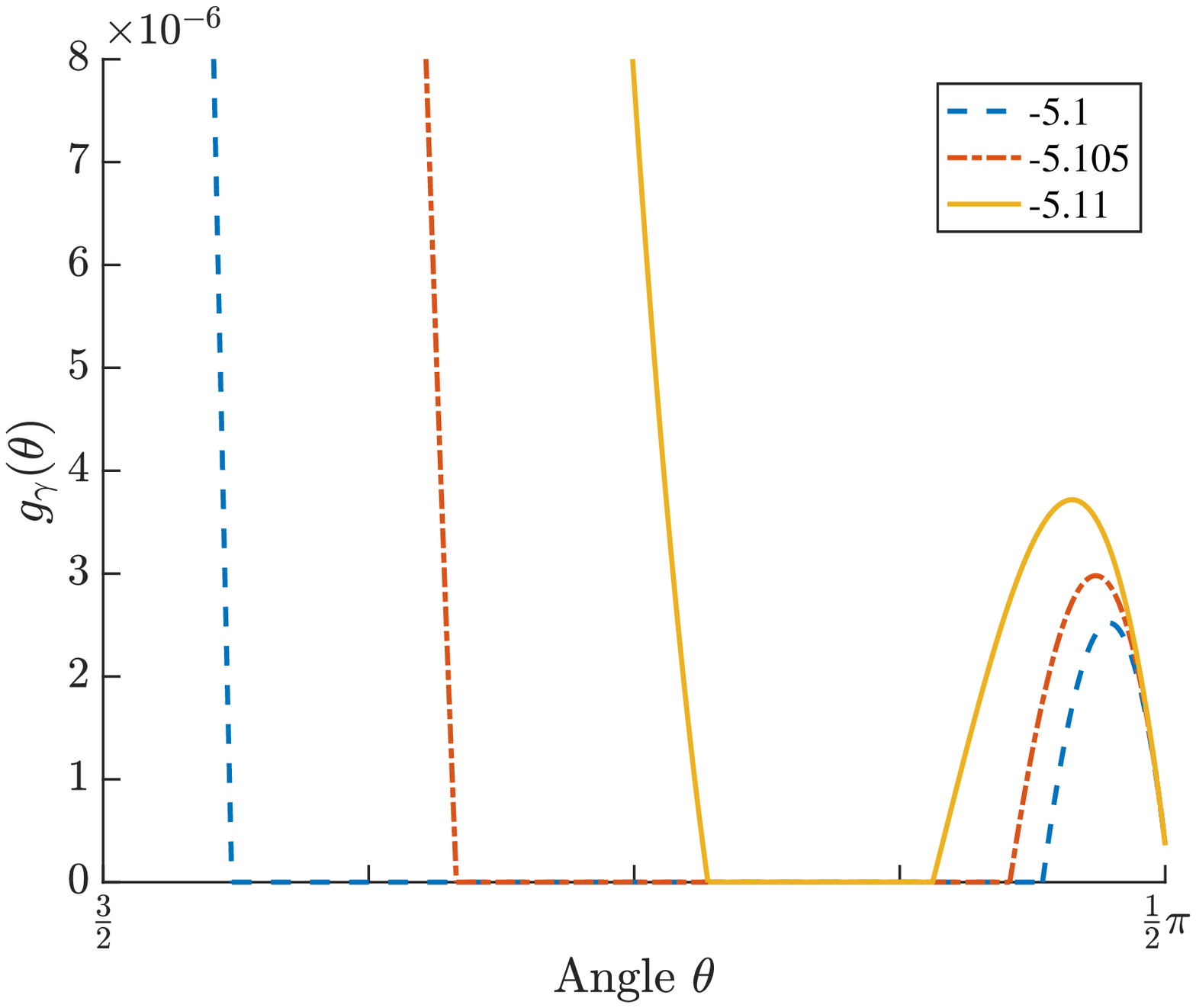} 
\label{fig:kcd_zoom}
}
\caption{The top left pane shows a contour plot of the level sets 
(in $\log_{10}$ scale, with label $k$ denoting $10^k$) of the 
objective function in \eqref{eq:kinv_cont} for a continuous-time example, with $z = x+\imagunit y$.
As this matrix is real, only the upper right quadrant of the complex plane is shown.
The global minimizer of  \eqref{eq:kinv_cont} lies in the small boxed area near $(x,y) \approx (0,6)$;
an enlarged view of this region is shown in the bottom left pane.
Contours are shown for $k=-3,-4,-5$ (dotted), $k = -5.1$ (solid), $k = -5.105$ (solid, unlabeled at top left),
and $k=-5.11$ (solid, not visible at top left).
For each of the three solid contours, 
$\sgfn(\theta)$ for $\gamma = 10^k$ is plotted in the right panes,
for the respective regions shown in the left panes.
For each angle tick mark in the right panes, the corresponding ray from the origin is shown as a dashed line in the  left panes.
It is easy to see the correspondence between the level sets for $k \in \{-5.1,-5.105,-5.11\}$ in the left panes and where their associated functions $\sgfn(\theta)$ 
are zero in the right panes.
}
\label{fig:kinv_cont_ex}
\end{figure}

In Figure~\ref{fig:kinv_cont_ex}, we show plots of $\sgfn(\theta)$ for different values of 
$\gamma$ for the $10 \times 10$ continuous-time example used in \cite[Section~8]{Mit19}.
The example is based on a demo from EigTool \cite{eigtool},
specifically $A = B - \kappa I$, where  $B = \texttt{companion\_demo(10)}$
and $\kappa = 1.001\alpha(B)$.
Since this matrix is real-valued, the level sets of $g(r,\theta)$ 
are symmetric with respect to the real axis, and so it is only necessary to sweep the upper right quadrant
of the complex plane, i.e., the domain of $\sgfn(\theta)$ can be reduced to $[0,\tfrac{\pi}{2})$.

Although we do not know of an analytic way of finding zeros of $\sgfn(\theta)$,
it is a continuous function of one real variable on a fixed finite interval
which we can approximate via interpolation.
Interpolation-based approximation of $\sgfn(\theta)$ is practical
as $\sgfn(\theta)$ is a rather well behaved on its finite domain and is relatively cheap to evaluate.
Moreover, even though $\sgfn(\theta)$ may be nondifferentiable at some points, 
modern interpolation software is adept at approximating functions that are nonsmooth and 
even discontinuous or have singularities.
Thus, as finding roots (and extrema) of polynomial or piecewise-polynomial interpolants is easy, approximating $\sgfn(\theta)$ via interpolation allows a way to find where $\sgfn(\theta)=0$ holds,
and in turn, find $\gamma$-level set points of $g(r,\theta)$.
Moreover, as we are about to explain, often a high-fidelity approximation for $\sgfn(\theta)$ 
is only needed once $\gamma \approx \Kinv$ holds, i.e., once our new method has converged.
Finally, note that even without interpolation, zeros of $\sgfn(\theta)$ may be found via sampling,
since by \cref{thm:props_kc}, $\mu(\sgset(\gamma)) > 0$ must hold if  $\gamma > \Kinv$.

\begin{algfloat}[!t]
\begin{algorithm}[H]
\floatname{algorithm}{Algorithm}
\caption{Interpolation-based Globality Certificate Algorithm}
\label{alg:interp}
\begin{algorithmic}[1]
	\REQUIRE{  
		$A \in \C^{n \times n}$ (nonnormal, $\alpha(A) \leq 0$, and $0 \not\in \Lambda(A)$) and $z_0 \in \C$ with $\Re z_0 > 0$.
		}
	\ENSURE{ 
		$\gamma^{-1} \approx \Kcon$ (continuous-time).
		\\ \quad
	}

	\STATE $\mathcal{D} \gets (-\tfrac{\pi}{2},\tfrac{\pi}{2})$
	\IF { $A$ is real }
		\STATE $\mathcal{D} \gets [0,\tfrac{\pi}{2})$
	\ENDIF
	\WHILE { true } 
		\STATE $\gamma \gets $ computed locally/globally minimal value of \eqref{eq:kinv_cont} initialized from $z_0$
		\STATE \COMMENT{Begin approximating $\sgfn(\theta)$ to check convergence or find new starting points}
		\STATE $p_\gamma(\theta) \gets 1$ \COMMENT{Initial guess for interpolant $p_\gamma(\theta)$ for 
				approximating $\sgfn(\theta)$}
		\WHILE { $p_\gamma(\theta)$ does not sufficiently approximate $\sgfn(\theta)$ on $\mathcal{D}$} 
			\STATE $[\theta_1,\ldots,\theta_l] \gets$ new sample points from $\mathcal{D}$
			\STATE \COMMENT{If new starting points are detected, restart optimization to lower $\gamma$:}
			\IF { $\sgfn(\theta_j) = 0$ for some $j \in \{1,\ldots,l\}$ }
				\STATE $z_0 \gets$ a point $r \e^{\imagunit \theta_j}$ such that $\imagunit r \in \Lambda(M,N_{\theta_j})$ 
						defined in \eqref{eq:eigMN_kc} with $r > 0$
				\STATE \textbf{goto line 6} \COMMENT{Restart optimization from $z_0$}
			\ENDIF
			\STATE \COMMENT{Otherwise, no starting points detected, keep improving $p_\gamma(\theta)$:}
			\STATE $p_\gamma(\theta) \gets$ improved interpolant of 
					$\sgfn(\theta)$ via $\theta_1,\ldots,\theta_l$
		\ENDWHILE
		\STATE \COMMENT{$p_\gamma(\theta)$ approximates $\sgfn(\theta)$ well and no new starting points were encountered} 
		\STATE \COMMENT {However, do  a final check before asserting that $\sgfn(\theta)$ has no other zeros:} 
		\STATE $[\theta_1,\ldots,\theta_l] = \argmin p_\gamma(\theta)$
		\IF { $\sgfn(\theta_j) = 0$ for some $j \in \{1,\ldots,l\}$ }
			\STATE $z_0 \gets$ a point $r \e^{\imagunit \theta_j}$ such that $\imagunit r \in \Lambda(M,N_{\theta_j})$ 
					defined in \eqref{eq:eigMN_kc} with $r > 0$
			\STATE \textbf{goto line 6} \COMMENT{Restart optimization from $z_0$}
		\ELSE
			\RETURN	\COMMENT{$p_\gamma(\theta) \approx \sgfn(\theta)$ and $\quad \Longrightarrow \quad \gamma \approx \Kinv$}
		\ENDIF
	\ENDWHILE
\end{algorithmic}
\end{algorithm}
\vspace{-0.3cm}
\algnote{
For simplicity of the pseudocode, we assume that optimization 
converges to local/global minimizers exactly and $z_0$ computed in lines 13 and 23
is never a stationary point of \eqref{eq:kinv_cont}.
Lines 7--19 describe the core of the interpolation-based globality certificate, where
we assume the interpolation process for approximating $\sgfn(\theta)$ is done via
some reliable method, e.g., Chebfun.
In lines 20--27, where a final check is done before asserting convergence, 
one can additionally/alternatively compute the roots $\{\theta_1,\ldots,\theta_l\}$ of $p_\gamma(\theta)$
and check the value of $\sgfn(\theta)$ at $0.5(\theta_j + \theta_{j+1})$ for all \mbox{$j=1,\ldots,l-1$}.
}
\end{algfloat}

In~\cref{alg:interp}, we provide pseudocode for our new approach to computing continuous-time
$\Kcon$ using optimization-with-restarts and our interpolation-based globality certificates, 
the latter of which we now describe at a high level.  
Our certificates assume that existing interpolation software can approximate $\sgfn(\theta)$ 
to essentially machine precision.
Given $\gamma \geq \Kinv$, our certificate works by beginning to sample~$\sgfn(\theta)$ for
various values of $\theta$ in order to approximate it on $\mathcal{D}$, or 
just $[0,\tfrac{\pi}{2})$ if $A$ is real.
Since interpolation methods are adaptive, this sampling happens in batches, where
$\sgfn(\theta)$ can be evaluated at the requested sample points in an ``embarrassingly parallel" manner.
If any zeros of $\sgfn(\theta)$ are encountered during a given batch of sampling,
then $\gamma$-level-set points of $g(r,\theta)$ have been detected and the interpolation process
is immediately halted.  Then the locations of detected level-set points are computed via \cref{thm:MN_kc},
and these are used to restart optimization in order to find a better (lower) minimizer;
for brevity, we assume that the detected level-set points are not exactly stationary.  
In this case, sampling $\sgfn(\theta)$ and restarting optimization suffices to lower $\gamma$ closer to $\Kinv$,
hence the interpolation process begins anew with the updated version of $\sgfn(\theta)$.
We now consider when interpolation produces a high-fidelity approximation $p_\gamma(\theta)$ for $\sgfn(\theta)$
but without ever encountering zeros during sampling.
To assert that $\gamma = \Kinv$ really holds, the interpolant $p_\gamma(\theta)$ is used to check
if $\sgfn(\theta)= 0 $ on regions that were not sampled.
This is possible to do since, by assumption, $p_\gamma(\theta)$ approximates $\sgfn(\theta)$ to machine precision.  
Thus, the global minimizer(s) of the interpolant $p_\gamma(\theta)$ are computed and used to check 
if $\sgfn(\theta) = 0$ at these angles.  
If this yields \emph{newly detected} non-stationary level-set points
(since this may simply recover known minimizers that were computed in the last round of optimization),
then optimization is restarted to lower $\gamma$ further.
If still no roots of $\sgfn(\theta)$ are discovered, then the roots of $p_\gamma(\theta)$ are computed and
are similarly used to check if $\sgfn(\theta) = 0$ holds elsewhere, specifically by evaluating $\sgfn(\theta)$
at the midpoints between consecutive roots.
Again, if new level-set points are detected, optimization is restarted from them.
Otherwise, our certificate has built a high-fidelity approximation to $\sgfn(\theta)$
and asserts that it cannot find non-stationary points in the $\gamma$-level set.
As $\mu(G(\theta) > 0$ must hold if $\gamma > \Kinv$ by \cref{thm:props_kc}, 
the algorithm concludes with $\gamma = \Kinv$.

Since the main cost of evaluating $\sgfn(\theta)$ is computing the spectrum of $(M,N_\theta)$,
\cref{alg:interp} has a work complexity of $\bigO(kn^3)$ and a memory complexity of $\bigO(n^2)$,
where $k$ is the total number of function values of $\sgfn(\theta)$ incurred (over all values of $\gamma$ encountered).  
The cost of finding minimizers of \eqref{eq:kinv_cont}, the other major component of \cref{alg:interp},
can be ignored, as (quasi-)Newton methods only require a handful of iterations to converge (for \emph{two}-variable problems),
while evaluating the function value and gradient/Hessian of the objective function in \eqref{eq:kinv_cont}
can be done with at most  $\bigO(n^3)$ work and $\bigO(n^2)$ memory; for more details, see our comments in the introduction 
on this.
As we show in the experiments, when $\gamma > \Kinv$,
very few samples are needed before a restart occurs.
Meanwhile, when $\gamma = \Kinv$,
the number of interpolation points needed to build a high-fidelity approximation
to $\sgfn(\theta)$ is not necessarily dependent on $\n$, and in fact, often acts more like a constant,
albeit a large one.  
The combination of $k$ not being too large and that 
only a single high-fidelity approximation to $\sgfn(\theta)$
is typically needed means that our interpolation-based globality certificates can be orders of magnitude faster
than earlier techniques based on solving fewer but \emph{much} larger eigenvalue problems.
Moreover, using parallel processing for the sampling phases only improves upon this
already large performance difference.
Finally, in stark contrast to all but one of the methods discussed in \S\ref{sec:intro_dtu} and \S\ref{sec:intro_kreiss},
our new level-set approach does not crucially rely on any \emph{single} computation for correctness.
The only way our new interpolation-based
approach can fail to restart optimization is if rounding errors prevent 
detection of level-set points for \emph{every} sampled root of $\sgfn(\theta)$;
 as $\mu(G(\theta)) > 0$ holds when $\gamma > \Kinv$, this seems quite unlikely
and so our new approach is more numerically reliable than previous ones.

\begin{remark}
Note that our interpolation-based globality certificates have two key differences 
to the supervised techniques discussed in the introduction for estimating Kreiss constants.
The first and more important difference is that a global maximizer of \eqref{eq:k1d_cont}
may be anywhere in $[0,\infty)$ and may occur on a very fast time scale,
which can make finding such maximizers very difficult.
Here, $\sgfn(\theta)$ is defined on \emph{the fixed finite interval} $(-\tfrac{\pi}{2},\tfrac{\pi}{2})$,
and its zeros form a subset with \emph{positive measure} when \mbox{$\gamma > \Kinv$}.
Hence finding zeros of $\sgfn(\theta)$ should be substantially easier
than finding global maximizers of \eqref{eq:k1d_cont}.
Second, $\sgfn(\theta)$ is more reliable to compute and cheaper to obtain;
computing $\aleps(A)$ via the criss-cross algorithms of \cite{BurLO03}
or \cite{BenM19} often involves computing all eigenvalues of several $2n \times 2n$ matrices.
\end{remark}

\begin{remark}
Certainly our certificate function defined in \eqref{eq:dist_kc} is not the only 
possible choice but one might wonder why we did not choose something simpler, e.g.,
an indicator function.  The reason is that if $\sgfn(\theta)$ were to return a 
fixed positive value whenever the associated ray does not intersect the level set,
then interpolation software may erroneously conclude 
 with very few sample points that the function is constant.
 This is because the error between the interpolant and $\sgfn(\theta)$
 would be exactly zero if none of the interpolation points happen to fall in $\sgset(\gamma)$,
 which may be small when $\gamma$ is close to $\Kinv$.
 Defining $\sgfn(\theta)$ so that it generally varies with $\theta$
 helps to ensure that the function is sufficiently sampled.
\end{remark}

\subsection{Efficient and robust evaluation of $\sgfn(\theta)$}
\label{sec:fast_eval}
By using an sHH structure-preserving eigensolver to compute $\Lambda(M,N_\theta)$,
imaginary eigenvalues will have exactly zero real part, 
and so roots of $\sgfn(\theta)$ should generally be computed as exact roots, 
i.e., $\sgfn(\theta)$ should be exactly zero numerically 
whenever $\theta$ corresponds to a ray intersecting the $\gamma$-level set.
While this is clearly appealing, 
as mentioned in \S\ref{sec:level_sets_kc}, the downside of structure preservation is that it involves solving 
a generalized eigenvalue problem, which if $N_\theta$ is nonsingular, is many times slower than solving the 
equivalent standard eigenvalue problem $N_\theta^{-1}M$.
However, in our new algorithm, the vast majority of evaluations for $\sgfn(\theta)$ 
will be for values that are nowhere close to being roots.  
This is because at $\gamma = \Kinv$, we expect that $\mu(\sgset(\theta))=0$.
Furthermore, if $\gamma > \Kinv$, $\mu(\sgset(\theta)) > 0$ holds, and so some rounding
error in the computed eigenvalues can typically be tolerated in obtaining roots of $\sgfn(\theta)$.
Consequently, the increased numerical robustness from using structure-preserving eigensolvers
is actually often not relevant in our new certificates.  With this in mind, we propose the following way 
to evaluate $\sgfn(\theta)$ much faster while still maintaining numerical reliability.

Given $\gamma \geq \Kinv$ and $\tilde \theta \in \mathcal{D}$,
if $N_{\tilde \theta}$ is singular, then $\sgfn(\tilde \theta)$ must be evaluated by computing the spectrum of the matrix pencil $(M,N_{\tilde \theta})$, so in this case there is little reason not to use a structure-preserving eigensolver.
However, if $N_{\tilde \theta}$ is nonsingular, then $\sgfn(\tilde \theta)$ is initially evaluated via computing the eigenvalues of $N_{\tilde \theta}^{-1}M$,
which as we have established in \S\ref{sec:level_sets_kc}, 
is not an issue as $\kappa(N_\theta)$ is typically small.
Given some small tolerance $\texttt{tol} > 0$,
if $\sgfn(\tilde \theta)$ is not attained by an eigenvalue $\lambda$ 
such that $\min \{ | \Im \lambda |, | \lambda |\} \leq \texttt{tol}$, i.e., $\lambda$ is deemed not
too close to the positive imaginary axis, then $\sgfn(\tilde \theta)$ 
can be considered to have been computed with sufficient accuracy to assert that angle $\tilde \theta$ is indeed not 
a root. 
Otherwise, the eigenvalues near the positive imaginary axis are, via \cref{thm:MN_kc},
used to check if level-set points to restart optimization have been detected.  
If so, then this is sufficient.
The only case that remains is that eigenvalues near the positive imaginary axis have been computed
but level-set points have not been detected.  As not detecting level-set points \emph{could} be the result of rounding errors,
$\Lambda(M,N_{\tilde \theta})$ is now recomputed using a structure-preserving eigensolver
to either overcome any rounding errors or verify that indeed $\sgfn(\tilde \theta) > 0$ holds.

As we expect $\mu(G(\theta)) = 0$ to hold once $\gamma = \Kinv$,
only a small minority evaluations of $\sgfn(\theta)$ should require the additional computation
with the structure-preserving eigensolver.  As such, the overall running time of our new algorithm
should be much faster than if the structure-preserving eigensolver was always used,
and by construction, numerical reliability remains uncompromised.

\section{A new approach for computing discrete-time Kreiss constants}
\label{sec:kdisc}
We now adapt our new globality certificates to compute discrete-time Kreiss constants to arbitrary accuracy,
i.e., to find global minimizers of \eqref{eq:kinv_disc}.
To do this, we will adapt \cref{alg:interp} and develop a new interpolation-based globality certificate
for discrete-time $\Kcon$.
In this discrete-time setting, a polar parametrization is used for both finding (feasible) minimizers of \eqref{eq:kinv_disc} (see \cite[Section~3.2]{Mit19} for details) and the interpolation-based globality certificate itself.
Thus, consider
\beq
	\label{eq:h_fns}
	h(r,\theta) \coloneqq \smin(H(r,\theta)) 
	\quad \text{and} \quad 
	H(r,\theta) \coloneqq \frac{r\eit I - A}{r - 1},
\eeq
so 
\[ 
	\Kinv = \inf_{r > 0, \, \theta \in (-\pi,\pi]} \, h(r,\theta).
\]
To create a discrete-time $\Kcon$ analogue of $\sgfn(\theta)$,
we make the following assumptions.
If $A$ is normal or $\rho(A) > 1$, computing $\Kcon$ is trivial, so we assume that neither condition holds.
Also, while in the previous section $\sgfn(\theta)$ required that $0 \not\in \Lambda(A)$,
our new certificate for discrete-time $\Kcon$ requires that $\gamma^2 \not\in \Lambda(AA^*)$.

\subsection{Level sets of $h(r,\theta)$ and another 1D radial level-set test}
As a key part of interpolation-based globality certificates is a 1D radial level-set test,
we begin with an analogue of \cref{thm:MN_kc}.
\begin{theorem}
\label{thm:ST_kd}
Let $A \in \C^{n\times n}$ and $\gamma, r, \theta \in \R$ with $r \neq 1$.
Then $\gamma \geq 0$ is a singular value of $H(r,\theta)$ defined in \eqref{eq:h_fns} 
if and only if $\imagunit r$ is an eigenvalue of the skew-Hamiltonian-Hamiltonian 
matrix pencil $(S,T_\theta)$, where
\beq
	\label{eq:eigST_kd}
	S \coloneqq
	\begin{bmatrix}
		A  		& -\gamma I \\
		\gamma I 	& -A^*
	\end{bmatrix}
	\quad \text{and} \quad
	T_\theta \coloneqq
	\begin{bmatrix}
		-\imagunit \eit I 			& \imagunit \gamma I  \\
	 	- \imagunit \gamma I    	& \imagunit \emit I
	\end{bmatrix},
\eeq
$T_\theta$ is singular if and only if $|\gamma| = 1$, and $(S,T_\theta)$ is regular if $|\gamma | \neq 1$.
\end{theorem}
\begin{proof}
It is easy to verify that $S$ and $T_\theta$ are respectively Hamiltonian and skew-Hamiltonian,
and so $(S,T_\theta)$ is an sHH matrix pencil.
Furthermore, the determinant of $T_\theta$ is simply $\det(T_\theta) = 1 - \gamma^2$.
Thus, $| \gamma | \neq 1$, i.e., $T_\theta$ being nonsingular, is
sufficient for $(S,T_\theta)$ to be a regular matrix pencil.
Now suppose $\gamma$ is a singular value of $H(r,\theta)$ with left and right singular vectors $u$ and $v$.
Then the following two equations hold:
\[
	\gamma
	\begin{bmatrix}
		u \\ v
	\end{bmatrix}
	=
	\begin{bmatrix}
		H(r,\theta)  	& 0 \\
		0 			& H(r,\theta)^*
	\end{bmatrix}
	\begin{bmatrix}
		v \\ u
	\end{bmatrix}
	\ \ \text{and} \ \
	\gamma (r-1)
	\begin{bmatrix}
		u \\ v
	\end{bmatrix}
	=
	\begin{bmatrix}
		r\eit I - A  	& 0 \\
		0 		& r\emit I - A^*
	\end{bmatrix}
	\begin{bmatrix}
		v \\ u
	\end{bmatrix}.
\]
Rearranging terms, this is equivalent to
\[
	\begin{bmatrix}
		A  	& 0 \\
		0 	& A^*
	\end{bmatrix}
	\begin{bmatrix}
		v \\ u
	\end{bmatrix}
	-\gamma 
	\begin{bmatrix}
		u \\ v
	\end{bmatrix}
	=
	r
	\begin{bmatrix}
		\eit I  	&	0 \\
		0 		& \emit I
	\end{bmatrix}
	\begin{bmatrix}
		v \\ u
	\end{bmatrix}
	-r\gamma 
	\begin{bmatrix}
		u \\ v
	\end{bmatrix}.
\]
Combining terms and multiplying the bottom block row by $-1$,
we equivalently have
\[
	\begin{bmatrix}
		A  		& -\gamma I \\
		\gamma I  & -A^*
	\end{bmatrix}
	\begin{bmatrix}
		v \\ u
	\end{bmatrix}
	=
	r
	\begin{bmatrix}
		\eit I  	& -\gamma I \\
		\gamma I	& -\emit I
	\end{bmatrix}
	\begin{bmatrix}
		v \\ u
	\end{bmatrix}.
\]
Noting that the matrix on the right multiplied by $-\imagunit$ is $T_\theta$ 
completes the proof.
\end{proof}

The point of Remark~\ref{rem:min_sv}, with appropriate substitutions, 
similarly applies to Theorem~\ref{thm:ST_kd}, $h(r,\theta)$, and $H(r,\theta)$,
and \cref{thm:ST_kd} also allows a way to show that the $\gamma$-level set
of $h(r,\theta)$ is bounded for $\gamma \in [0,1)$.

\begin{theorem}
\label{thm:bounded_disc}
Let $A \in \C^{n\times n}$ and $\gamma \in [0,1)$.  The $\gamma$-level set of $h(r,\theta)$ defined in \eqref{eq:h_fns} is bounded.
Moreover, if $\rho(A) < 1$, the $\gamma$-level set is compact.
\end{theorem}
\begin{proof}
The proof follows similarly to the proof of \cref{thm:bounded_cont}, with the key part being
that  the eigenvalues of $(S,T_\theta)$ are finite for all $\theta$ (since $T_\theta$ is nonsingular if $\gamma \neq \pm 1$), 
and so $\max_{\theta \in \R} \rho(S,T_\theta)$ must be finite.
\end{proof}

As with our continuous-time certificate using $\sgfn(\theta)$ and $(M,N_\theta)$, 
for discrete-time $\Kcon$ we need to ensure that zero cannot be an eigenvalue of $(S,T_\theta)$,
the precise conditions for which are given by the following result.

\begin{theorem}
\label{thm:zero_kd}
Let $A \in \C^{n\times n}$ and $\gamma, \theta \in \R$.
Then the matrix pencil $(S,T_\theta)$ defined by \eqref{eq:eigST_kd} has zero as an eigenvalue 
if and only if the matrix $AA^*$ has $\gamma^2$ as an eigenvalue.
Consequently, $\gamma^2 \not\in \Lambda(AA^*)$ also ensures that $(S,T_\theta)$ is regular.
\end{theorem}
\begin{proof}
As the blocks of $S$ are all $n \times n$ and 
the lower two blocks $\gamma I$ and $-A^*$ commute,
the if-and-only-if equivalence holds because 
$\det(S) = \det(-AA^* + \gamma^2 I)$.
Lastly, if $0 \not \in \Lambda(S)$, then clearly $(S,T_\theta)$ must be a regular matrix pencil.
\end{proof}

Finally, we consider the condition number of $T_\theta$ and when 
computing the eigenvalues of $(S,T_\theta)$ via $T_\theta^{-1}S$ is possible.
The following result shows that $\kappa(T_\theta) = \kappa(N_\theta)$, i.e., 
$T_\theta$ will generally be very well conditioned for all relevant values of $\gamma$,
hence using $T_\theta^{-1}S$ to compute $\Lambda(S,T_\theta)$ is not problematic.

\begin{theorem}
Let $A \in \C^{n \times n}$ and $\gamma,\theta \in \R$ with $\gamma \neq \pm 1$. 
Then the spectrum of the matrix pencil $(S,T_\theta)$ defined by \eqref{eq:eigST_kd} is 
equal to the spectrum of 
\beq
	\label{eq:eigS_kd}
	S_\theta \coloneqq T_\theta^{-1} S = 
	\frac{\imagunit}{1 - \gamma^2}
	\begin{bmatrix}
	\emit A - \gamma^2 I 	& \gamma(A^* - \emit I)  \\
	\gamma(A - \eit I )  	& \eit A^* - \gamma^2 I
	\end{bmatrix}
\eeq
and if $\gamma \in [0,1)$, then $\max_{\theta \in \R} \kappa(T_\theta) = \tfrac{1+\gamma}{1 - \gamma}$.
\end{theorem}
\begin{proof}
The matrix given in \eqref{eq:eigS_kd} simply follows by using the explicit form of $S_\theta^{-1}$,
which exists if and only if $\gamma \neq \pm 1$.
Applying \cref{lem:condnum} to $T_\theta$ with $\gamma \in [0,1)$,
$\kappa(T_\theta) = \tfrac{1 + \gamma}{1 - \gamma }$, thus completing the proof.
\end{proof}

\subsection{An interpolation-based globality certificate for $h(r,\theta)$}
For \eqref{eq:kinv_disc}, we correspondingly consider sweeping the entire complex via a ray from the origin
to see where it intersects the $\gamma$-level set of $h(r,\theta)$ outside of the closed unit disk.
Thus, we construct a new continuous function $\shfn : (-\pi,\pi] \mapsto [0,\pi^2]$ similar to \eqref{eq:dist_kc}:
\bseq
	\label{eq:gamma_kd}
	\begin{align}
	\label{eq:dist_kd}
	\shfn(\theta) &\coloneqq
	\min \{ \Arg(-\imagunit \lambda)^2 : \lambda \in \Lambda(S,T_\theta), \lambda \not\in [0,\imagunit], \Re \lambda \leq 0\}, \\				 
	\label{eq:set_kd}
	\shset(\gamma) &\coloneqq
	\interior \{ \theta : \shfn(\theta) = 0, \, \theta \in (-\pi,\pi] \},	
	\end{align}
\eseq
similarly keeping in mind that $\Lambda(S,T_\theta)$ always has imaginary-axis symmetry,
regardless of whether or not the level sets of $h(r,\theta)$ have symmetry.

\begin{theorem}
\label{thm:props_kd}
Let $A \in \C^{n \times n}$ with $\rho(A) \leq 1$.
Then for any $\gamma \geq 0$ such that $\gamma^2 \not\in \Lambda(AA^*)$, 
the function $\shfn(\theta)$ defined in \eqref{eq:dist_kd} has the following properties:
\begin{enumerate}[label=\roman*.]
\item $\shfn(\theta) \geq 0$ for all $\theta \in \mathcal{D} \coloneqq (-\pi,\pi]$,
\item $\shfn(\theta) = 0$ if and only if there exists $\imagunit r \in \Lambda(S,T_\theta)$ with $r \in \R$ and $r > 1$,
\item $\shfn(\theta)$ is continuous on its entire domain $\mathcal{D}$, 
\item $\shfn(\theta)$ is differentiable at a point $\theta$ if the eigenvalue $\lambda \in \Lambda(S,T_\theta)$ attaining the value of $\shfn(\theta)$ is unique and simple.
\end{enumerate}
Furthermore,  the following properties hold for the set $\shset(\gamma)$ defined in \eqref{eq:set_kd}:
\begin{enumerate}[label=\roman*.,resume]
\item if $\Kinv < \gamma$, then 
	$0 < \mu(\shset(\gamma))$,
\item $\gamma_1 \leq \gamma_2$ if and only if $\mu (\shset(\gamma_1)) \leq \mu (\shset(\gamma_2))$,
\item if $\gamma > 1$, then $\mu(\shset(\gamma)) = 2\pi$,
\end{enumerate}
where $\mu(\cdot)$ is the Lebesgue measure on $\R$.
\end{theorem}
\begin{proof}
The proof mostly follows the proof of \cref{thm:props_kc}, 
now using \cref{thm:ST_kd,thm:zero_kd} instead of \cref{thm:MN_kc,thm:zero_kc}.
The notable differences are as follows.
The second property requires the exclusion of any imaginary eigenvalues of 
$(S,T_\theta)$ that are also in the interval $[0,\imagunit]$, per the definition
of $\shfn(\theta)$ given in \eqref{eq:dist_kd}.
This key change keeps $\shfn(\theta)$ strictly positive 
whenever $(S,T_\theta)$ has one or more eigenvalues on the imaginary axis in $[0,\imagunit]$ but not in $(\imagunit,\infty)$.
The continuity property is unaffected by this exclusion but does require 
our assumption that $\gamma^2 \not\in \Lambda(AA^*)$,
which by \cref{thm:zero_kd} guarantees that zero is never an eigenvalue of $(S,T_\theta)$ for any $\theta \in \R$.
For the properties of $\shset(\gamma)$, the main differences to note are that $\lim_{r\to \infty} h(r,\theta) = 1$ for any $\theta$,
while $\lim_{r\to 1^+} h(r,\theta) = \infty$ for almost all $\theta$.  
This second limit can only be finite for at most $n$ values of $\theta \in \mathcal{D}$,
namely at angles corresponding to unimodular eigenvalues of $A$.
\end{proof}

In Figure~\ref{fig:kinv_disc_ex}, we show plots of $\shfn(\theta)$ for different values of 
$\gamma$ for the $10 \times 10$ discrete-time example used in \cite[Section~8]{Mit19},
namely $A = \tfrac{1}{13}B + \tfrac{11}{10}I$, where $B = \texttt{convdiff\_demo(11)}$ from EigTool.
As this matrix is real, the level sets of $h(r,\theta)$ are symmetric with respect to the real axis,
and so only the upper half of the complex plane is shown.

\def\ddfull{$\shfn(\theta)$ on $\left[0,\pi\right]$}
\def\ddzoom{$\shfn(\theta)$ on $\left[0,\tfrac{3}{8}\pi\right]$}

\begin{figure}[!t]
\centering
\subfloat[Level sets]{
\includegraphics[scale=\imgscale,trim={0.9cm 0.05cm 1.5cm 0.8cm},clip]{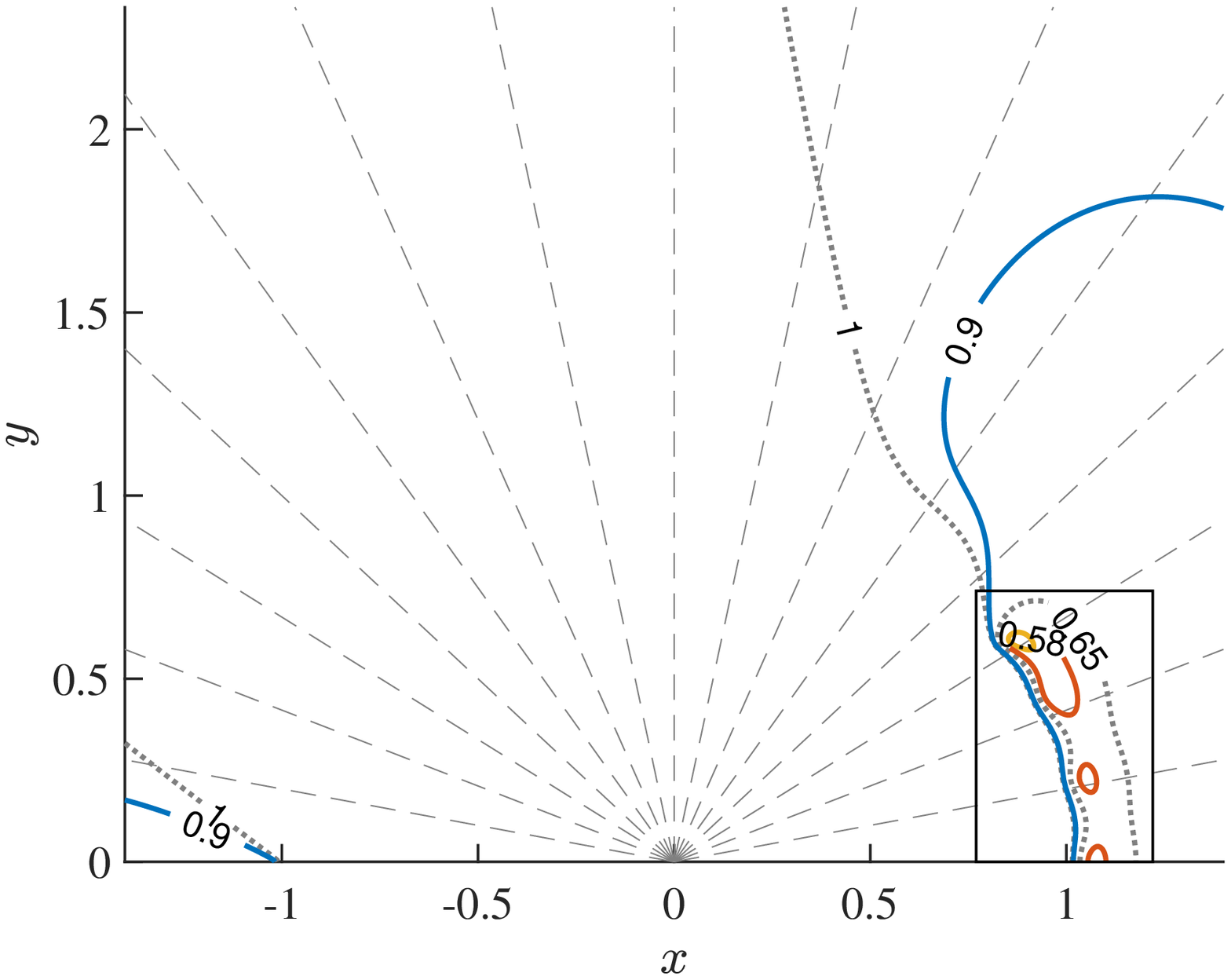} 
\label{fig:kdl}
} 
\subfloat[\ddfull]{
\includegraphics[scale=\imgscale,trim={0.3cm 0.05cm 1.5cm 0.8cm},clip]{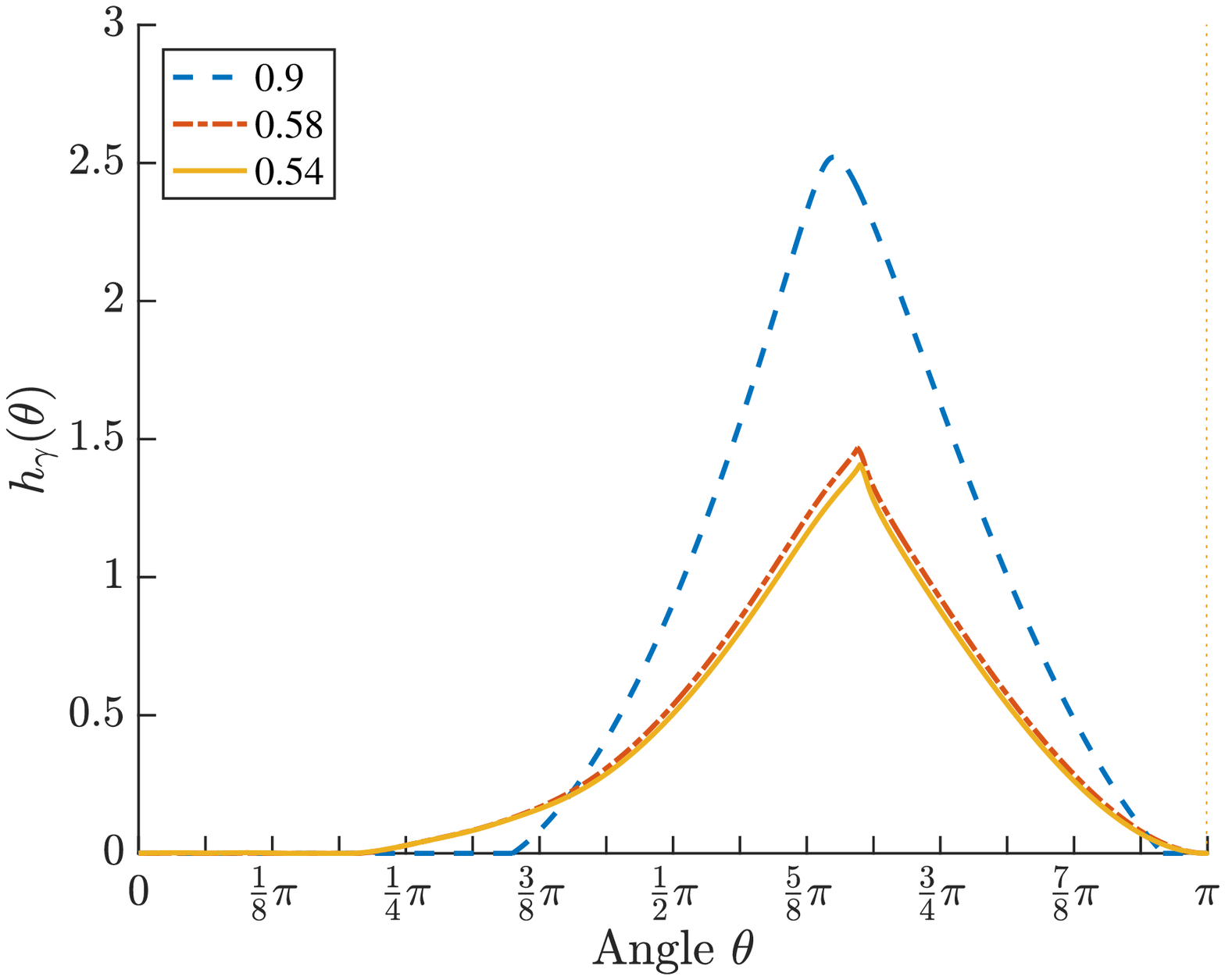} 
\label{fig:kdd}
}
\\
\subfloat[Level sets (enlarged view)]{
\includegraphics[scale=\imgscale,trim={0.9cm 0.05cm 1.3cm 0.3cm},clip]{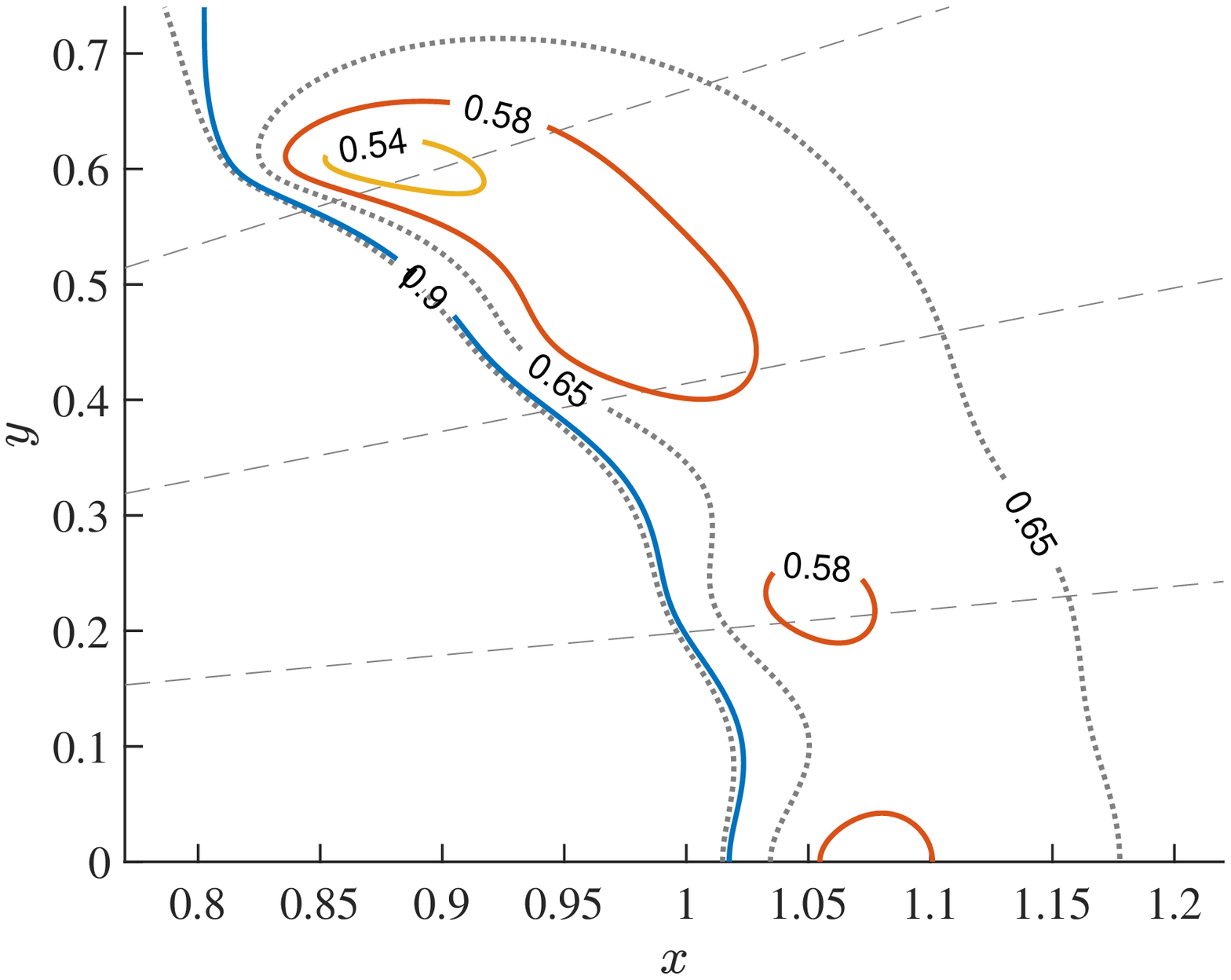} 
\label{fig:kdl_zoom}
}
\subfloat[\ddzoom]{
\includegraphics[scale=\imgscale,trim={0.3cm 0.05cm 1.5cm 0.3cm},clip]{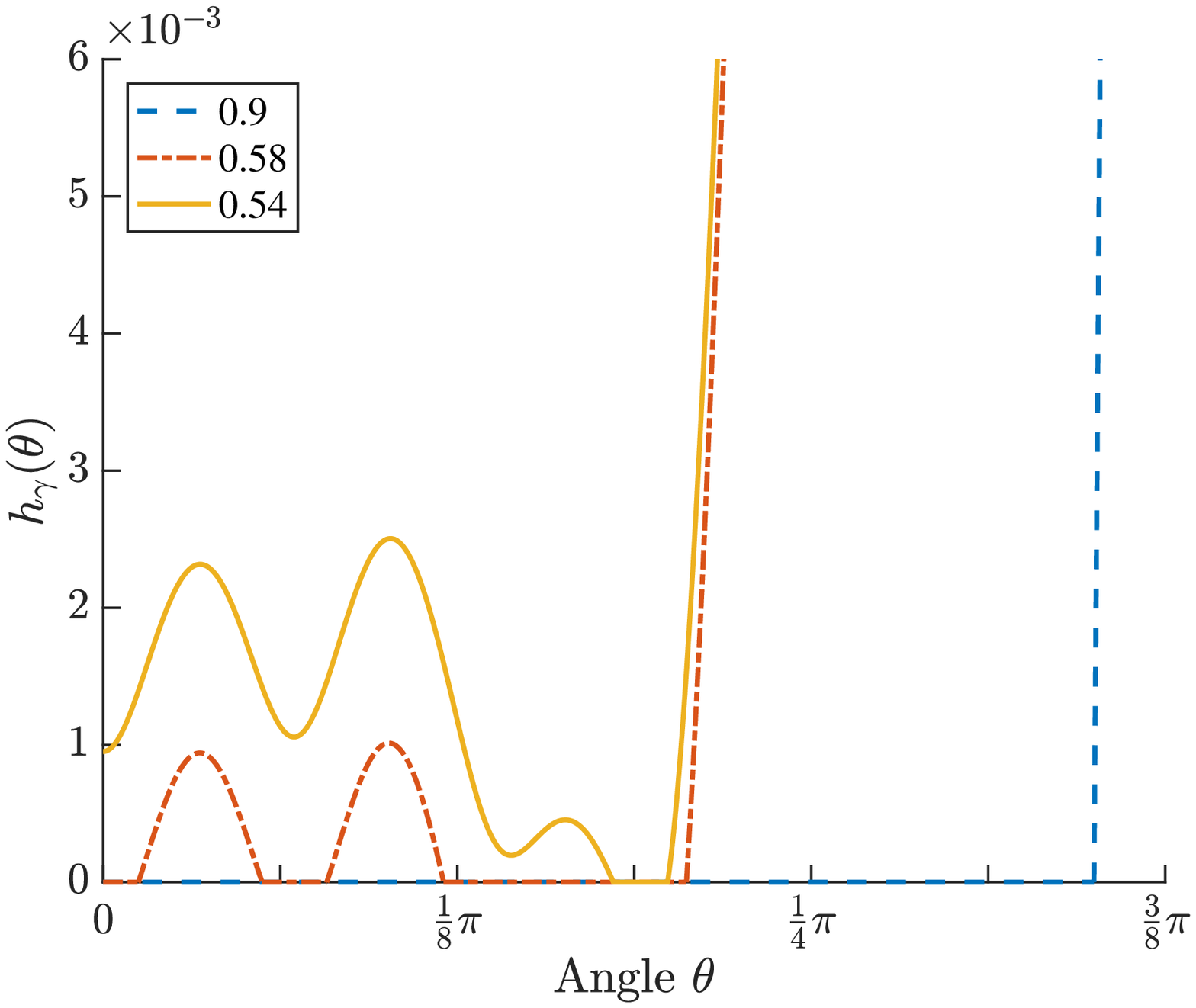} 
\label{fig:kdd_zoom}
}
\caption{The top left pane shows a contour plot of the level sets (now in linear scale) of the 
objective function in \eqref{eq:kinv_disc} for a discrete-time example, with $z = x+\imagunit y$.
As this matrix is real, only the upper half of the complex plane is shown.
The global minimizer of  \eqref{eq:kinv_disc} lies in the boxed area, in the well near the top left corner;
an enlarged view of this region is shown in the bottom left pane.
Contours are shown for $\gamma=1$ (dotted), $\gamma = 0.9$ (solid), $\gamma = 0.65$ (dotted), 
$\gamma = 0.58$ (solid), and $\gamma = 0.54$ (solid, not visible at top left).
For each of the solid contours, 
the corresponding $\shfn(\theta)$ function is plotted in the right panes,
for the respective regions shown in the left panes.
For each angle tick mark in the right panes, the corresponding ray from the origin is shown as a dashed line in the left panes.
The correspondence between the $\gamma$-level set and where $\shfn(\theta)=0$ 
is clearly evident.  In the top right pane, the discontinuity in $\shfn(\theta)$ 
for $\gamma = 0.54$ near $\theta = \pi$ is due 
to excluding eigenvalues of $(S,T_\theta)$ that lie inside the eccentric ellipse $\delta^{-2}x^2 + y^2 = 1$ with $\delta=10^{-8}$.
}
\label{fig:kinv_disc_ex}
\end{figure}

The removal of eigenvalues from $\Lambda(S,T_\theta)$ on the imaginary axis in the interval $[0,\imagunit]$ 
for $\shfn(\theta)$ requires further comment.
In fact, any eigenvalue of $(S,T_\theta)$
inside the closed unit disk is irrelevant, but per \cref{rem:min_sv}, \cref{thm:ST_kd} may detect 
level-set points of $\sigma_k(zI - A)/|(|z| - 1)|$, for any $k = 1,\ldots,n$, where $\sigma_k(\cdot)$
is the $k$th largest singular value.
However, excluding all eigenvalues in the unit disk could introduce discontinuities, as  $\Arg(\cdot)$ may jump
if an eigenvalue of $(S,T_\theta)$ enters or exits the unit disk.
Instead, by only excluding those in $[0,\imagunit]$, continuity is preserved, and 
while $\shfn(\theta)$ may become infinitesimal as
eigenvalues inside the unit disk may be arbitrarily close to $[0,\imagunit]$, they cannot 
introduce a zero of $\shfn(\theta)$ precisely since $[0,\imagunit]$ itself is excluded.
In practice, using a structure-preserving eigensolver for $(S,T_\theta)$ means that computed
eigenvalues on the imaginary axis will be exactly imaginary, and so removing any that are in $[0,\imagunit]$ can be done exactly.
When a structure-preserving eigensolver is not used, 
two other measures can help deal with rounding errors in the real parts of computed eigenvalues.
First, in lines~13 and 23 for the adaptation of \cref{alg:interp} (see \cref{sec:alg_disc} for the complete details), we can restrict to eigenvalues $\imagunit r$ of $(S,T_\theta)$ with $r > 1$ so that optimization is restarted only if \emph{feasible} nonstationary
points are found; for $\shfn(\theta)$, this crucial specification ensures that
$\shfn(\theta)$ is increasingly better approximated by interpolant $p_\gamma(\theta)$ until either $p_\gamma(\theta) \approx \shfn(\theta)$ or a detected zero of $\shfn(\theta)$ \emph{also} leads to detection of a level-set point outside the unit disk.
Second, to help ensure that $\shfn(\theta) = 0$ only if $r\eit$ with $r > 1$ is a level-set point, 
we can discard any eigenvalue of $(S,T_\theta)$ that lies inside
the eccentric ellipse defined by $\delta^{-2}x^2 + y^2 = 1$ for some small $\delta > 0$.  
Technically, this may still introduce discontinuities in $\shfn(\theta)$,
but they are much less likely to occur than when excluding the entire unit disk ($\delta = 1$).
Note that such a discontinuity can be seen in Figure~\ref{fig:kinv_disc_ex}; see the caption for details.

\subsection{Adapting \cref{alg:interp} for discrete-time $\Kcon$}
\label{sec:alg_disc}
The following modifications of \cref{alg:interp} are needed for 
computing discrete-time $\Kcon$.
For input, it is assumed that $A$ is nonnormal with $\rho(A) \leq 1$ and $z_0 \in \C$ with $|z_0| > 1$.
While $\shfn(\theta)$ requires that $\gamma^2 \not\in \Lambda(AA^*)$, per \cref{thm:zero_kd},
this is a very mild assumption; 
clearly there only up to $n$ values of $\gamma \geq 0$ such that $\gamma^2 \in \Lambda(AA^*)$, and in 
the unlikely case that one of these is encountered,  
simply perturbing $\gamma$ by a slight amount would suffice.
In lines 1--3, $\mathcal{D}$ should be initially set to $(-\pi,\pi]$ and reduced to $[0,\pi]$
if $A$ is real. 
Throughout the pseudocode and accompanying note, \eqref{eq:kinv_cont} and $\sgfn(\theta)$
should be replaced by \eqref{eq:kinv_disc} and $\shfn(\theta)$, respectively,
and $0.5(\theta_l + (\theta_1 + 2\pi))$ should also be included
when doing the additional check described in the note.
As alluded to earlier, in lines 13 and 23, 
``$\imagunit r \in \Lambda(M,N_{\theta_j})$ defined in \eqref{eq:eigMN_kc} with $r > 0$"
should be replaced with 
``$\imagunit r \in \Lambda(S,T_{\theta_j})$  defined in \eqref{eq:eigST_kd} with $r > 1$".
For increased efficiency, 
$\shfn(\theta)$ should be evaluated in analogous manner as described in \S\ref{sec:fast_eval}
for $\sgfn(\theta)$ but in this case, it is only necessary to consider recomputing the eigenvalues 
of $T_\theta^{-1}S$ via its matrix pencil form when a computed eigenvalue is within a distance 
\texttt{tol} of the interval $[\imagunit,\infty)$ on the imaginary axis (as opposed to the interval $[0,\infty)$).
These modifications for discrete-time $\Kcon$ do not alter 
the $\bigO(kn^3)$ work complexity and $\bigO(n^2)$ memory characteristics.

\section{A new approach for computing the distance to uncontrollability}
\label{sec:dtu}
We now turn to adapting our new globality certificates for $\dtu$, i.e.,
to find global minimizers of \eqref{eq:dtu}.
For the optimization phases, local minimizers of \eqref{eq:dtu} 
should be found using Cartesian coordinates, with a quasi-Newton method or Newton's method for fast local convergence; 
the gradient of  \eqref{eq:dtu} can be found in \cite[p.~358]{BurLO04}, while the corresponding Hessian can 
be obtained via a straightforward modification to the derivation of the Hessian of \eqref{eq:kinv_cont} in \cite[Section~3.1]{Mit19}.
For our interpolation-based globality certificate for $\dtu$, we again use polar coordinates:
\beq
	\label{eq:f_fns}
	f(r,\theta) \coloneqq \smin(F(r,\theta)) 
	\quad \text{and} \quad 
	F(r,\theta) \coloneqq \begin{bmatrix} A - r\eit I & B \end{bmatrix},
\eeq
so 
\[ 
	\dtu = \inf_{r > 0, \, \theta \in (-\pi,\pi]} \, f(r,\theta).
\]
Our upcoming $\dtu$ analogue of $\sgfn(\theta)$ and $\shfn(\theta)$ requires that 
$\gamma^2 \not\in \Lambda(AA^* + BB^*)$, a condition which can be easily ensured, as we will explain.

\subsection{Level sets of $f(r,\theta)$ and another 1D radial level-set test}
We again begin with a result enabling a radial 1D level-set test.
Our following result is essentially a modified version of the result derived in \cite[eq.~(2.4)]{Gu00}.
As noted there, similar results were also previously developed in \cite[Theorem~3.1]{Bye90a} and \cite[Lemmas~2.1 and 2.2]{GaoN93}.  
The key difference here, besides some simplifications, 
is that we derive an sHH matrix pencil 
for the desirable imaginary-axis symmetry of its eigenvalues.  The proof is also a bit different.

\begin{theorem}
\label{thm:CD_dtu}
Let $A \in \C^{n\times n}$ and $\gamma, r, \theta \in \R$ with $\gamma \neq 0$.
Then $\gamma > 0$ is a singular value of $F(r,\theta)$ defined in \eqref{eq:f_fns} if and only if $\imagunit r$ is an eigenvalue of the 
regular skew-Hamiltonian-Hamiltonian 
matrix pencil $(C,D_\theta)$, where
\beq
	\label{eq:eigCD_dtu}
	C \coloneqq
	\begin{bmatrix}
		A   			& \widetilde B \\
		\gamma I 		& -A^*  		
	\end{bmatrix}
	\quad \text{and} \quad
	D_\theta \coloneqq
	\begin{bmatrix}
		-\imagunit \eit I 	& 0 \\
		0			& \imagunit \emit I 
	\end{bmatrix},
\eeq
$\widetilde B \coloneqq \tfrac{1}{\gamma}BB^* - \gamma I$, and $D_\theta$ is always nonsingular.
\end{theorem}
\begin{proof}
Clearly $D_\theta$ is always nonsingular and it easy to verify that it is also skew-Hamiltonian and
$C$ is Hamiltonian; hence, $(C,D_\theta)$ is a regular sHH matrix pencil.
Now suppose $\gamma$ is a singular value of $F(r,\theta)$ with left and right singular vectors 
$u$ and $v = \begin{bsmallmatrix} v_1 \\ v_2 \end{bsmallmatrix}$.
Then the following two equations hold:
\[
	\gamma
	\begin{bmatrix}
		u \\ v
	\end{bmatrix}
	=
	\begin{bmatrix}
		F(r,\theta)  	& 0 \\
		0 			& F(r,\theta)^*
	\end{bmatrix}
	\begin{bmatrix}
		v \\ u
	\end{bmatrix}
	\ \text{and} \
	\gamma
	\begin{bmatrix}
		u \\ v_1 \\ v_2
	\end{bmatrix}
	=
	\begin{bmatrix}
		A - r\eit I  	& B 	& 0 \\
		0 		& 0 	& A^* - r\emit I \\
		0 		& 0 	& B^*
	\end{bmatrix}
	\begin{bmatrix}
		v_1 \\ v_2 \\ u
	\end{bmatrix}.
\]
The last block row yields $v_2 = \tfrac{1}{\gamma}B^* u$,
and so by making this substitution in the equation above, 
we equivalently have 
\[
	\gamma
	\begin{bmatrix}
		u \\ v_1 
	\end{bmatrix}
	=
	\begin{bmatrix}
		A - r\eit I  	& \tfrac{1}{\gamma}BB^*  	& 0 \\
		0 		& 0 					& A^* - r\emit I
	\end{bmatrix}
	\begin{bmatrix}
		v_1 \\ u \\ u
	\end{bmatrix}
	=
	\begin{bmatrix}
		A - r\eit I  	& \tfrac{1}{\gamma}BB^* \\
		0 		& A^* - r\emit I
	\end{bmatrix}
	\begin{bmatrix}
		v_1 \\ u 
	\end{bmatrix}.
\]
Rearranging terms to separate out the terms involving $r$ 
and multiplying the resulting lower block row by $-1$ yields
\[
	\begin{bmatrix}
		A  	& \tfrac{1}{\gamma}BB^* \\
		0 	& -A^*
	\end{bmatrix}
	\begin{bmatrix}
		v_1 \\ u
	\end{bmatrix}
	-
	\gamma
	\begin{bmatrix}
		u \\ -v_1 
	\end{bmatrix}
	=
	\begin{bmatrix}
		A  		& \widetilde B \\
		\gamma I  & -A^*
	\end{bmatrix}
	\begin{bmatrix}
		v_1 \\ u
	\end{bmatrix}
	=
	r
	\begin{bmatrix}
		\eit I 	& 0 \\
		0 	& -\emit I 
	\end{bmatrix}
	\begin{bmatrix}
		v_1 \\ u 
	\end{bmatrix}.
\]
Noting that the matrix on the right multiplied by $-\imagunit$ is $D_\theta$ 
completes the proof.
\end{proof}

Remark~\ref{rem:min_sv}, with appropriate substitutions, 
also applies to \cref{thm:CD_dtu}, $f(r,\theta)$, and $F(r,\theta)$.
While \cref{thm:CD_dtu} can also be used to show that the $\gamma$-level
set of $f(r,\theta)$ is compact for any finite $\gamma$, this is not necessary 
as it is already well known; see \cite[p.~353]{BurLO04}.
A third argument for this is via \cite[Theorem~2.2]{WriT02} and equivalently 
considering $\smin(\begin{bsmallmatrix} A - zI & B\end{bsmallmatrix}^*)$,
whose lower level sets are rectangular pseudospectra.

\begin{theorem}
\label{thm:bounded_dtu}
Let $A \in \C^{n\times n}$, $B \in \C^{n \times m}$, and $\gamma \geq 0$. 
The $\gamma$-level set of $f(r,\theta)$ defined in \eqref{eq:f_fns} is compact.
\end{theorem}

While in the Kreiss constant setting it is clear when the level sets of $g(r,\theta)$ and $h(r,\theta)$
are symmetric with respect to the real axis, the symmetry conditions for $f(r,\theta)$
are slightly more nuanced.

\begin{theorem}
\label{thm:dtu_symmetry}
Let $A \in \C^{\n \times \n}$, $B \in \C^{\n \times m}$, and let $\sigma_k(\cdot)$ 
denote the $k$th singular value. 
Then 
$
	\sigma_k \left( \begin{bmatrix} A - \lambda I & B \end{bmatrix} \right) 
	= 
	\sigma_k \left( \begin{bmatrix} A - \overline \lambda I & B \end{bmatrix} \right)
$
if either \emph{(i)} $A$~and~$B$ are both real-valued matrices or \emph{(ii)} $A$ is Hermitian.
\end{theorem}
\begin{proof}
Case (i) holds since conjugation does not change singular values, i.e.,
\[
	\sigma_k \left( \begin{bmatrix} A - \lambda I & B \end{bmatrix} \right) 
	= 
	\sigma_k \left( \begin{bmatrix} \, \overline A - \overline \lambda I & \overline B \, \, \end{bmatrix} \right) 
	=
	\sigma_k \left( \begin{bmatrix} A - \overline \lambda I & B \end{bmatrix} \right),
\]
where the middle equivalence uses the fact that $A = \overline A$ and $B = \overline B$ since both are real.
Case (ii) follows from the equivalence $\sigma_k(M) \Longleftrightarrow \sigma_k^2 \in \Lambda(MM^*)$:
\begin{align*}
	\sigma_k \left( \begin{bmatrix} A - \lambda I & B \end{bmatrix} \right) 
	\quad &\Longleftrightarrow \quad 
	\sigma_k^2 \in \Lambda \left( \begin{bmatrix} A - \lambda I & B \end{bmatrix} 
				\begin{bmatrix} A^* - \overline \lambda I \\ B^* \end{bmatrix} \right) \\
	&\Longleftrightarrow \quad 
	\sigma_k^2 \in \Lambda \left( \begin{bmatrix} AA^* - \lambda A^* - \overline \lambda A + |\lambda|^2 I + BB^* \end{bmatrix} \right) \\
	&\Longleftrightarrow \quad 
	\sigma_k^2 \in \Lambda \left( \begin{bmatrix} AA^* - \lambda A - \overline \lambda A^* + |\lambda|^2 I + BB^* \end{bmatrix} \right) \\
	 &\Longleftrightarrow \quad 
	\sigma_k^2 \in \Lambda \left( \begin{bmatrix} A - \overline \lambda I & B \end{bmatrix} 
				\begin{bmatrix} A^* - \lambda I \\ B^* \end{bmatrix} \right) \\
	&\Longleftrightarrow \quad 
	\sigma_k \left( \begin{bmatrix} A - \overline \lambda I & B \end{bmatrix} \right).
\end{align*}
where the third line uses the assumption that $A = A^*$.
\end{proof}

We now consider the conditions under which zero is an eigenvalue of $(C,D_\theta)$,
since our interpolation-based globality certificates require excluding this possibility.

\begin{theorem}
\label{thm:zero_dtu}
Let $A \in \C^{n\times n}$, $B \in \C^{\n \times m}$, and $\gamma, \theta \in \R$ with $\gamma \neq 0$.
Then the matrix pencil $(C,D_\theta)$ defined by \eqref{eq:eigCD_dtu} has zero as an eigenvalue 
if and only if the matrix $AA^* + BB^*$ has $\gamma^2$ as an eigenvalue.
\end{theorem}
\begin{proof}
Since the blocks of $C$ are all square matrices of the same size and 
the lower two blocks $\gamma I$ and $-A^*$ commute,
the if-and-only-if equivalence holds because
\begin{align*}
	\det(C) = \det(-AA^* - \gamma\widetilde B) 
	&= \det(-AA^* - \gamma(\tfrac{1}{\gamma}BB^* - \gamma I))\\
	&= \det(-AA^* - BB^* + \gamma^2 I)).
\end{align*}
\end{proof}

As the next result states, it is clear that $D_\theta^{-1}C$ can be used to 
compute the eigenvalues of $(C,D_\theta)$ when forgoing structure-preserving eigensolvers.

\begin{theorem}
Let $A \in \C^{n\times n}$, $B \in \C^{\n \times m}$, and $\gamma, \theta \in \R$.
The condition number of $D_\theta$, $\kappa(D_\theta)$, equals one for any $\theta$, and 
the spectrum of matrix pencil $(C,D_\theta)$ defined by \eqref{eq:eigCD_dtu} 
is equal to the spectrum of 
\beq
	\label{eq:eigM_dtu}
	C_\theta \coloneqq D_\theta^{-1}C =
	\imagunit
	\begin{bmatrix}
	\emit A 			& \emit \widetilde B \\
	-\gamma \eit I 		& \eit A^*
	\end{bmatrix}.
\eeq
\end{theorem}
\begin{proof}
The proof is an immediate consequence of the facts that $D_\theta$ is diagonal and all its diagonal entries are nonzero for any $\theta$.
\end{proof}

\subsection{An interpolation-based globality certificate for $f(r,\theta)$}
Unlike for Kreiss constants, there are no domain restrictions for where a minimizer
of \eqref{eq:dtu} may lie, so our $\dtu$ certificate must sweep the entire complex plane with a ray from the origin
to find level-set points.
Moreover, the origin can be immediately tested by simply evaluating $f(0,\theta)$ for any $\theta$.
Thus, given $\gamma > 0$, we construct another continuous function $\sffn : (-\pi,\pi] \mapsto [0,\pi^2]$ similar to $\sgfn(\theta)$:
\bseq
	\label{eq:gamma_dtu}
	\begin{align}
	\label{eq:dist_dtu}
	\sffn(\theta) &\coloneqq
	\min \{ \Arg(-\imagunit \lambda)^2 : \lambda \in \Lambda(C,D_\theta), \Re \lambda \leq 0\}, \\				 
	\label{eq:set_dtu}
	\sfset(\gamma) &\coloneqq
	\interior \{ \theta : \sffn(\theta) = 0, \, \theta \in (-\pi,\pi] \},	
	\end{align}
\eseq
noting that $\Lambda(C,D_\theta)$ always has imaginary-axis symmetry.

\begin{theorem}
\label{thm:props_dtu}
Let $A \in \C^{n \times n}$ and $B \in \C^{n \times m}$.
Then for any $\gamma > 0$ such that $\gamma^2 \not\in \Lambda(AA^* + BB^*)$, 
the function $\sffn(\theta)$ defined in \eqref{eq:dist_dtu} has the following properties:
\begin{enumerate}[label=\roman*.]
\item $\sffn(\theta) \geq 0$ for all $\theta \in \mathcal{D} \coloneqq (-\pi,\pi]$,
\item $\sffn(\theta) = 0$ if and only if there exists $\imagunit r \in \Lambda(C,D_\theta)$ with $r \in \R$ and $r > 0$,
\item $\sffn(\theta)$ is continuous on its entire domain $\mathcal{D}$, 
\item $\sffn(\theta)$ is differentiable at a point $\theta$ if the eigenvalue $\lambda \in \Lambda(C,D_\theta)$ attaining the value of $\sffn(\theta)$ is unique and simple.
\end{enumerate}
Furthermore,  the following properties hold for the set $\sfset(\gamma)$ defined in \eqref{eq:set_dtu}:
\begin{enumerate}[label=\roman*.,resume]
\item if $\dtu < \gamma$, then 
	$0 < \mu(\sfset(\gamma))$,
\item $\gamma_1 \leq \gamma_2$ if and only if $\mu (\sfset(\gamma_1)) \leq \mu (\sfset(\gamma_2))$,
\item if $\gamma > f(0,\theta)$ for any $\theta \in \R$, then $\mu(\sfset(\gamma)) = 2\pi$,
\end{enumerate}
where $\mu(\cdot)$ is the Lebesgue measure on $\R$.
\end{theorem}
\begin{proof}
This proof also follows the proof of \cref{thm:props_kc}, 
now using \cref{thm:CD_dtu,thm:zero_dtu} instead of \cref{thm:MN_kc,thm:zero_kc}.
Here the continuity  property of $\sffn(\theta)$ requires our assumption that $\gamma^2 \not\in \Lambda(AA^*+BB^*)$,
which by \cref{thm:zero_dtu} guarantees that zero is never an eigenvalue of $(C,D_\theta)$ for any $\theta \in \R$.
For $\sfset(\gamma)$, the corresponding arguments use
that $f(r,\theta)$ is continuous and $\lim_{r\to \infty} f(r,\theta) = \infty$ for any $\theta$.
\end{proof}

Note that by \cref{thm:zero_dtu,thm:CD_dtu}, 
our assumption that $\gamma^2$ is not an eigenvalue of $AA^* + BB^*$ is equivalent to 
$\gamma$ not being a singular value of $F(0,\theta)$ for any $\theta \in \R$.
As such, the properties of $\sffn(\theta)$ hold as long as $\gamma < f(0,\theta)$.
Since optimization-with-restarts monotonically decreases the value of $\gamma$ until it converges to $\dtu$,
we can easily guarantee that $\gamma^2$ is never an eigenvalue of $AA^* + BB^*$
just by initializing at the origin.  Provided the origin is not a stationary point, optimization 
guarantees finding a point $(\tilde r, \tilde \theta)$ such that $ f(\tilde r,\tilde \theta) < f(0,\theta)$.
Otherwise, either other starting points can be evaluated in order to find a function value
lower than $f(0,\theta)$ or the initial value of $\gamma$ can simply be set to slightly less than $f(0,\theta)$
before commencing the first certification computation.
Note that while $\sffn(\theta)$ is not defined for $\gamma = 0$, 
this is not a problem as there is no need to do a globality check when $f(r,\theta) = 0$, 
as $f(r,\theta)$ is never negative.
Finally, if $f(r,\theta) = f(r,-\theta)$, i.e., the level sets have real-axis symmetry, then the domain
$\mathcal{D}$ can be reduced to $[0,\pi]$.

For brevity, we forgo showing illustrative plots of $\sffn(\theta)$ here, but an example is shown later in Figure~\ref{fig:cert_kahan}.

\subsection{Adapting \cref{alg:interp} for $\dtu$}
We modify \cref{alg:interp} to
compute $\dtu$ as follows.
For input, $A  \in \C^{n \times n}$, $B \in \C^{n\times m}$, 
and $z_0 \in \C$ without restriction.
By also including the origin as an initial point,
ensuring $\gamma^2 \not\in \Lambda(AA^* + BB^*)$ only requires
that the origin not be stationary or, if no other starting points results in a value of $\gamma$
less than $f(0,\theta)$, that the initial value of $\gamma$ be set slightly less than $f(0,\theta)$.
In lines 1--3, $\mathcal{D}$ should be initially set to $(-\pi,\pi]$ and reduced to $[0,\pi]$
if the level sets have real-axis symmetry, per the conditions given in \cref{thm:dtu_symmetry}.
Throughout the pseudocode and accompanying note, \eqref{eq:kinv_cont} and $\sgfn(\theta)$
should be replaced by \eqref{eq:dtu} and $\sffn(\theta)$, respectively,
and $0.5(\theta_l + (\theta_1 + 2\pi))$ should also be included
when doing the additional check described in the note.
In lines 13 and 23, 
``$\imagunit r \in \Lambda(M,N_{\theta_j})$ defined in \eqref{eq:eigMN_kc} with $r > 0$"
should be replaced with 
``$\imagunit r \in \Lambda(C,D_{\theta_j})$  defined in \eqref{eq:eigCD_dtu} with $r > 0$".
For increased efficiency, 
$\sffn(\theta)$ should be evaluated in analogous manner to that described in \S\ref{sec:fast_eval}
for $\sgfn(\theta)$.
The $\bigO(kn^3)$ work complexity and $\bigO(n^2)$ memory characteristics again hold.

\section{Numerical experiments}
\label{sec:experiments}
To validate our new interpolation-based globality certificates, we implemented a proof-of-concept of \cref{alg:interp}
in \matlab\ and compared it against the existing state-of-the-art methods of \cite{Mit19} for continuous- and discrete-time Kreiss constants
and \cite{GuMOetal06} for the distance to uncontrollability.
Experiments were performed in MATLAB R2017b using a computer with two Intel Xeon Gold 6130 processors (16 cores each, 32 total) and 192GB of RAM.
The supplementary material includes code, test examples, and a detailed description of both our implementation and 
experimental setup for reproducibility of all results, tables, and figures in the paper;
for brevity, we only give essential details here.
We plan to add ``production-ready" implementations of \cref{alg:interp} for $\Kcon$ and $\dtu$
to a future release of ROSTAPACK~\cite{rostapack}.

For implementing \cref{alg:interp}, \texttt{fminunc} was used for finding local minimizers,
while $\sgfn(\theta)$, $\shfn(\theta)$, and $\sffn(\theta)$ were evaluated using \texttt{eig}.
Per \S\ref{sec:fast_eval} on efficient evaluation, 
our code first attempts to compute the values of $\sgfn(\theta)$, $\shfn(\theta)$, and $\sffn(\theta)$
via the standard eigenvalue problem formulations in \eqref{eq:eigM_kc}, \eqref{eq:eigS_kd}, and \eqref{eq:eigM_dtu}.
For simplicity, our prototype code only resorts to using the generalized eigenvalue problems in \eqref{eq:eigMN_kc}, \eqref{eq:eigST_kd}, and \eqref{eq:eigCD_dtu}
when infs, nans, or errors are encountered;
for a more robust implementation, eigenvalues of these sHH matrix pencils should be computed using a structure-preserving eigensolver such as \cite{BenBMetal02,BenSV16},
and this should also be done whenever the computed values of $\sgfn(\theta)$, $\shfn(\theta)$, and $\sffn(\theta)$ (when using \texttt{eig} and the standard eigenvalue problems) are close to zero.
For approximating $\sgfn(\theta)$, $\shfn(\theta)$, and $\sffn(\theta)$, we used Chebfun,
a sophisticated and efficient toolbox for 
``computing with functions to about 15-digit accuracy"\footnote{The quote is taken from the homepage of \url{http://www.chebfun.org}.}
that it is also adept at handling nonsmooth functions when its \texttt{splitting} option is enabled.
To replicate the design of \cref{alg:interp}, where optimization is restarted when zeros of $\sgfn(\theta)$, $\shfn(\theta)$, and $\sffn(\theta)$
are encountered, our code simply throws and catches an error in order to halt Chebfun; 
this allows us to immediately restart optimization without letting Chebfun finish building a high-fidelity approximation and requires no modifications
to Chebfun itself.
Our prototype attempts to restart optimization using one or more of the detected level-set points but not necessarily all;
a more robust implementation might first check whether or not any of these points are non-stationary before deciding to halt Chebfun early.
When Chebfun does build an approximation without encountering zeros, our prototype 
does the additional global convergence checks described in \cref{alg:interp} (lines 20--27 and its accompanying note).
Finally, our code terminates if none of the new starting points leads to a meaningful decrease in the estimate $\gamma$,
i.e., if the relative improvement in $\gamma$ is less than $10^{-14}$.
This additional test is necessary in practice as optimization software will generally
 not compute minimizers exactly and our interpolation-based globality certificates may 
still detect level-set points when a global minimizer has been found to numerical precision.

\begin{remark}
To demonstrate our interpolation-based globality certificates and encourage multiple restarts, 
we intentionally chose starting points such that only local, not global, minimizers would be found 
on the first round of optimization.
Moreover, our prototype performs a new globality certificate as soon as
optimization results in a relative decrease of $10^{-6}$ or more from the value of $\gamma$ for the preceding certificate.
As such, some detected level-set points may be ignored, which if used, could have led to larger decreases
in $\gamma$.
In practice, it will likely be more efficient to run optimization from all or at least many of the detected level-set points, 
to avoid making unnecessarily small updates.
For similar reasons, more than a single starting point should be used.  
\end{remark}

\begin{remark}
For conceptual simplicity, we have so far intentionally omitted a few other notable implementation details,
which we now briefly describe.
First, in Algorithm~\ref{alg:interp}, to account for rounding error, 
the interpolation-based globality certificates should not be done with the 
value of $\gamma$ computed in line~6 but rather $(1 - \texttt{tol})\gamma$
for a relative tolerance $\texttt{tol} \in (0,1)$, e.g., $\texttt{tol} = 10^{-14}$.
Second, for continuous-time $\Kcon$ and $\dtu$, when the level sets do not have real-axis symmetry,
it may be beneficial to shift the ``search point" from the origin, e.g., for $\Kcon$, one might instead 
use the average of the imaginary parts of the eigenvalues of $A$.  
Finally, for discrete-time $\Kcon$ and $\dtu$, there is an additional technique that
can often provide an additional factor-of-two speedup.
By noting that both $\shfn(\theta)$ and $\shfn(\theta + \pi)$ can be computed via 
a single computation of the spectrum of $(S,T_\theta)$,
our interpolation-based globality certificates can be computed by approximating 
$\min \{ \shfn(\theta), \shfn(\theta + \pi) \}$ over half of the domain $\mathcal{D}$,
and the same can be analogously done when computing $\dtu$.
For simplicity in the comparisons here, we forgo using this additional optimization. 
\end{remark}

\subsection{Comparisons to earlier methods}
\label{sec:comp}
Since we address parallel computation in \S\ref{sec:par}, 
here we consider a single-core evaluation of all the methods.  
We did this by calling \texttt{parpool(1)} in \matlab\ and by not using any \texttt{parfor} loops.
The test problems, whose dimensions are listed in \cref{table:overall}, are as follows.
For continuous-time $\Kcon$, \texttt{companion}~(stab.)\ is the stabilized EigTool example we used to generate \cref{fig:kinv_cont_ex},
while \texttt{boeing('S')} and \texttt{orrsommerfeld} are directly from EigTool.
For discrete-time $\Kcon$, \texttt{convdiff} (mod.)\ is the modified EigTool example we used to generate \cref{fig:kinv_disc_ex},
while \texttt{randn}~\#1~(stab.)\ and \texttt{randn}~\#2~(stab.)\ are randomly generated stable complex matrices, scaled so their spectral radii are 0.999.
While these discrete-time examples have very low Kreiss constants, they are useful for demonstration as $h(r,\theta)$ 
has multiple different local minima for each of them.
For $\dtu$, \cite[Section~4.3]{GuMOetal06} used real-valued examples generated by setting $A$ to different 
sizes of the \texttt{kahan} demo from EigTool and $B = \texttt{randn(}\n\texttt{,}m\texttt{)}$;
for our experiments here, we generated two such examples using larger values of $n$ and $m$, namely, 
\texttt{kahan}~($m=20$)\ and \texttt{kahan}~($m=30$).
Of the eight examples, \texttt{randn}~\#1~(stab.)\ and \texttt{randn}~\#2~(stab.)\ do not have level sets with real-axis symmetry,
while the others do.

\begin{table}[t]
\setlength{\tabcolsep}{3.84pt}
\centering
\small
\begin{tabular}{ l | rr | l | r | rr  } 
\toprule
\multicolumn{1}{c}{} & \multicolumn{1}{c}{}  & \multicolumn{1}{c}{} & \multicolumn{1}{c}{} & 
\multicolumn{3}{c}{Time (sec.)} \\ 
\cmidrule(lr){5-7}
\multicolumn{1}{l}{Problem} & 
	\multicolumn{1}{c}{$\n$} & 
	\multicolumn{1}{c}{$z_0$} &
	\multicolumn{1}{c}{Computed Value} & 
	\multicolumn{1}{c}{New} &
	\multicolumn{1}{c}{2D Fixed} &
	\multicolumn{1}{c}{2D Vari.} \\
\midrule
\multicolumn{5}{l}{$\Kcon$ (continuous)} & \multicolumn{2}{c}{Single LS Test}\\
\midrule
\texttt{companion} (stab.) & 10   & $6{+}6\imagunit$   & $1.29186707013556 \times 10^{5}  $ & $   0.5$ &        0.5 &        1.0 \\ 
\texttt{boeing('S')}       & 55   & $1{+}50\imagunit$  & $3.62541052800213 \times 10^{4}  $ & $   6.1$ &     6226.5 &     3446.7 \\ 
\texttt{orrsommerfeld}     & 100  & $10{+}10\imagunit$ & $3.93230474282055 \times 10^{1}  $ & $ 149.6$ &   170547.2 &   197426.0 \\ 
\midrule
\multicolumn{5}{l}{$\Kcon$ (discrete)} & \multicolumn{2}{c}{Single LS Test}\\
\midrule
\texttt{convdiff} (mod.)   & 10   & $-1{+}1\imagunit$  & $1.89501339090580 \times 10^{0}  $ & $   1.8$ &        0.9 &        0.9 \\ 
\texttt{randn} \#1 (stab.) & 50   & $1{+}1\imagunit$   & $1.75843606578311 \times 10^{0}  $ & $ 128.7$ &     3324.0 &     3248.0 \\ 
\texttt{randn} \#2 (stab.) & 100  & $1{-}1\imagunit$   & $2.35849495574647 \times 10^{0}  $ & $1223.8$ &    \multicolumn{2}{r}{--- \texttt{out-of-mem} ---} \\ 
\midrule
\multicolumn{5}{l}{$\dtu$} & \multicolumn{2}{c}{Full D\&C Alg.}\\
\midrule
\texttt{kahan} ($m=20$)    & 60   & $0{+}0\imagunit$   & $3.88211512261161 \times 10^{-2} $ & $  51.1$ &      246.2 & \multicolumn{1}{c}{---} \\ 
\texttt{kahan} ($m=30$)    & 150  & $0{+}0\imagunit$   & $1.82581469530120 \times 10^{-2} $ & $ 644.4$ &    27454.4 & \multicolumn{1}{c}{---} \\ 
\bottomrule
\end{tabular}
\caption{The eight problems tested.  The size of the matrix $A$ is given by $n$, while $z_0$ is the initial point used for the first round of optimization.
The  values of $\Kcon$ and $\dtu$ computed by \cref{alg:interp}  are given under ``Computed Value".
Elapsed wall-clock times (in seconds) are given in the three rightmost columns.
For  \cref{alg:interp}, the total running times are reported under ``New".
For Kreiss constants, rather than running the complete algorithms of \cite{Mit19},
we only recorded the time to perform a single 2D level-set test (``Single LS Test" in the table) for each problem.
Consequently, these times greatly underreport the actual costs to run the full algorithms of \cite{Mit19}.
As the methods of \cite{Mit19} either use fixed- or variable-distance 2D level-set tests, 
times are reported for both types, respectively under ``2D Fixed" and ``2D Vari.", except
for \texttt{randn} \#2 (stab.), where out-of-memory errors occurred.
For $\dtu$, for which only fixed-distance 2D level-set tests are relevant, 
the times to compute $\dtu$ using the complete divide-and-conquer method of \cite{GuMOetal06} (``Full D\&C Alg." in the table) are reported.
}
\label{table:overall}
\end{table}

For computing Kreiss constants, we compare the efficiency of \cref{alg:interp} with the earlier 2D level-set methods of \cite{Mit19},
using the code provided in the supplementary material of \cite{Mit19}.
However, the running times we report in \cref{table:overall} are \emph{not} for the complete algorithms of \cite{Mit19}, 
but rather just the time to perform a single 2D level-set test.
Recall that the methods of \cite{Mit19} always require performing at least one 2D~level-set test,
and these tests are the dominant cost, with $\bigO(n^6)$ work when using dense eigensolvers.
Thus, it suffices to time a single level-set test for each method of \cite{Mit19}.
We did not use the asymptotically faster divide-and-conquer versions from \cite{Mit19}, 
as they appear to be less reliable when computing Kreiss constants; see \cite[Section~8]{Mit19}.
For continuous-time $\Kcon$, the generalized eigenvalue problems that appear in the fixed- and variable-distance 2D level-set tests of \cite{Mit19}
were solved with \texttt{eig}, while the corresponding quadratic eigenvalue problems for the discrete-time $\Kcon$ tests
were solved with \texttt{polyeig}.
In \cref{table:overall}, we see that for the small ($n=10$) examples, \texttt{companion}~(stab.)\ and \texttt{convdiff}~(mod.),
the total running time of \cref{alg:interp} is comparable with the cost of a single 2D level-set of \cite{Mit19},
but our new approach is much faster for larger dimensions.
In fact, for the other continuous-time $\Kcon$ examples, \texttt{boeing('S')} ($n=55$) and \texttt{orrsommerfeld} ($n=100$), 
\cref{alg:interp} is generally over 1000 times faster 
than a single 2D level-set test.  For the two \texttt{randn}-based discrete-time $\Kcon$ examples, \cref{alg:interp} is roughly 
25~times faster than a single 2D level-set test for $n=50$, while 
it was not even possible to time the discrete-time 2D level-set tests for $n=100$, since \texttt{polyeig} immediately ran out of memory (on a computer with a 192 GB of RAM).
To compare accuracy of the methods, we only considered the two small examples (both $n=10$), 
due to the high cost of running the 2D level-set-based methods for larger $n$.
The $\Kcon$ estimates computed by \cref{alg:interp} in \cref{table:overall} for 
\texttt{companion}~(stab.)\ and  \texttt{convdiff}~(mod.)\ respectively agree to 11 and 15 digits 
to the corresponding values reported in \cite[Table~1]{Mit19} for the optimization-with-restart methods of \cite{Mit19}.
The slight discrepancy for  \texttt{companion}~(stab.)\ is almost certainly due to the fact that optimization solvers do not find minimizers exactly, 
and so there will generally be some variability in the least significant digits of $\gamma$.
This can likely be dealt with via tighter tolerances, using different optimization solvers, and/or using more starting points per restart.

For computing the distance to uncontrollability, we compared \cref{alg:interp} with the divide-and-conquer-based 
method of \cite{GuMOetal06}.  Since divide-and-conquer has an asymptotic work complexity of $\bigO(\n^4)$ on average and $\bigO(\n^5)$ in the worst case,
it was feasible to run the full method of \cite{GuMOetal06} on our test problems;
specifically, we compared against the \texttt{dist\_uncont\_hybrid} 
routine\footnote{Available at \url{http://home.ku.edu.tr/~emengi/software/robuststability.html}},
which uses BFGS for optimization and divide-and-conquer 2D level-set tests when 
\texttt{opts.method=1} and \texttt{opts.eig\_method=1} are set.
For the smaller \texttt{kahan}-based example ($n=60$, $m=20$), \cref{alg:interp}
is 4.8 times faster than  \texttt{dist\_uncont\_hybrid}, while for the larger example ($n=150$, $m=30$),
\cref{alg:interp} is 42.6 times faster.
We expect that this performance gap will generally widen more as $n$ increases.
The $\dtu$ estimates computed by \cref{alg:interp} for \texttt{kahan} ($m=20$)\ and \texttt{kahan} ($m=30$)\ respectively agreed to 12 and 13 digits
with those computed by \texttt{dist\_uncont\_hybrid}, with our new method returning the (slightly) smaller answers for both.

\begin{table}
\centering
\small
\begin{tabular}{ l | rc | rc | rc | rc } 
\toprule
\multicolumn{1}{c}{} & \multicolumn{8}{c}{\# of $\theta$'s evaluated per certificate and rel. diff. in $\gamma$} \\
\cmidrule(lr){2-9}
\multicolumn{1}{l}{Problem} & 
	\multicolumn{2}{c}{Restart 1} & 
	\multicolumn{2}{c}{Restart 2} & 
	\multicolumn{2}{c}{Restart 3} & 
	\multicolumn{2}{c}{Restart 4} \\
\midrule
\texttt{companion} (stab.) &            $15$  & \texttt{1e-02}  &   $\textbf{389}$ &            ---  &              --- &            ---  &              --- &            --- \\ 
\texttt{boeing('S')}       &            $15$  & \texttt{7e-01}  &            $15$  & \texttt{7e-01}  &   $\textbf{535}$ &            ---  &              --- &            --- \\ 
\texttt{orrsommerfeld}     &            $15$  & \texttt{9e-01}  &  $\textbf{3048}$ &            ---  &              --- &            ---  &              --- &            --- \\ 
\midrule
\texttt{convdiff} (mod.)   &            $15$  & \texttt{3e-01}  &            $15$  & \texttt{4e-02}  &            $31$  & \texttt{3e-02}  &  $\textbf{4084}$ & \texttt{1e-15} \\ 
\texttt{randn} \#1 (stab.) &            $63$  & \texttt{1e-01}  & $\textbf{12448}$ &            ---  &              --- &            ---  &              --- &            --- \\ 
\texttt{randn} \#2 (stab.) &            $15$  & \texttt{2e-01}  &            $15$  & \texttt{2e-01}  &           $127$  & \texttt{1e-01}  & $\textbf{18672}$ &            --- \\ 
\midrule
\texttt{kahan} ($m=20$)    &            $15$  & \texttt{7e-01}  &  $\textbf{3529}$ &            ---  &              --- &            ---  &              --- &            --- \\ 
\texttt{kahan} ($m=30$)    &            $15$  & \texttt{6e-01}  &            $15$  & \texttt{2e-01}  &  $\textbf{6246}$ &            ---  &              --- &            --- \\ 
\bottomrule
\end{tabular}
\caption{For each restart using our new interpolation-based globality certificates, 
the left number is the total number of points at which Chebfun evaluated $\sgfn(\theta)$, $\shfn(\theta)$, or $\sffn(\theta)$ 
for the current estimate $\gamma$ until either new starting points were found (which immediately restarts optimization) or Chebfun terminated on its own;
bold font indicates the last certificate computed.
The right number is the relative difference obtained by the next round of optimization to lower $\gamma$.
Note that for \texttt{convdiff}~(mod.), the last certificate actually produced new starting points 
but optimization was unable to meaningfully lower estimate $\gamma$ further, and so our code terminated after a round of 
optimization instead of after a certificate test.
}
\label{table:restarts}
\end{table}

In Table~\ref{table:restarts}, we show the number of points at which Chebfun evaluates $\sgfn(\theta)$, $\shfn(\theta)$, or $\sffn(\theta)$ (as appropriate)
for each interpolation-based globality certificate that is performed.
As can be seen, before a global minimizer is obtained, relatively few values of $\theta$ are evaluated by Chebfun
before new starting points are discovered and optimization commences again, demonstrating that high-fidelity 
approximations are indeed only needed once a global minimizer has been found.
Furthermore, as hoped, the number of function evaluations needed to build the final interpolants 
does not dramatically increase as the problems get larger.  
The number of functions evaluations is instead correlated with how complex  
$\sgfn(\theta)$, $\shfn(\theta)$, and $\sffn(\theta)$ are, which is not necessarily related to the problem dimension.
In Figure~\ref{fig:certs}, we plot $\sgfn(\theta)$, $\shfn(\theta)$, and $\sffn(\theta)$ for the final values of $\gamma$ computed by 
\cref{alg:interp} for three of our examples.

\def\scaledfn{0.34}
\begin{figure}[htbp]
\centering
\subfloat[\texttt{boeing('S')}: $\sgfn(\theta)$ in linear scale (left) and in $\log_{10}$ scale (right)]{
\includegraphics[scale=\scaledfn,trim={0.3cm 0.05cm 0.6cm 0.7cm},clip]{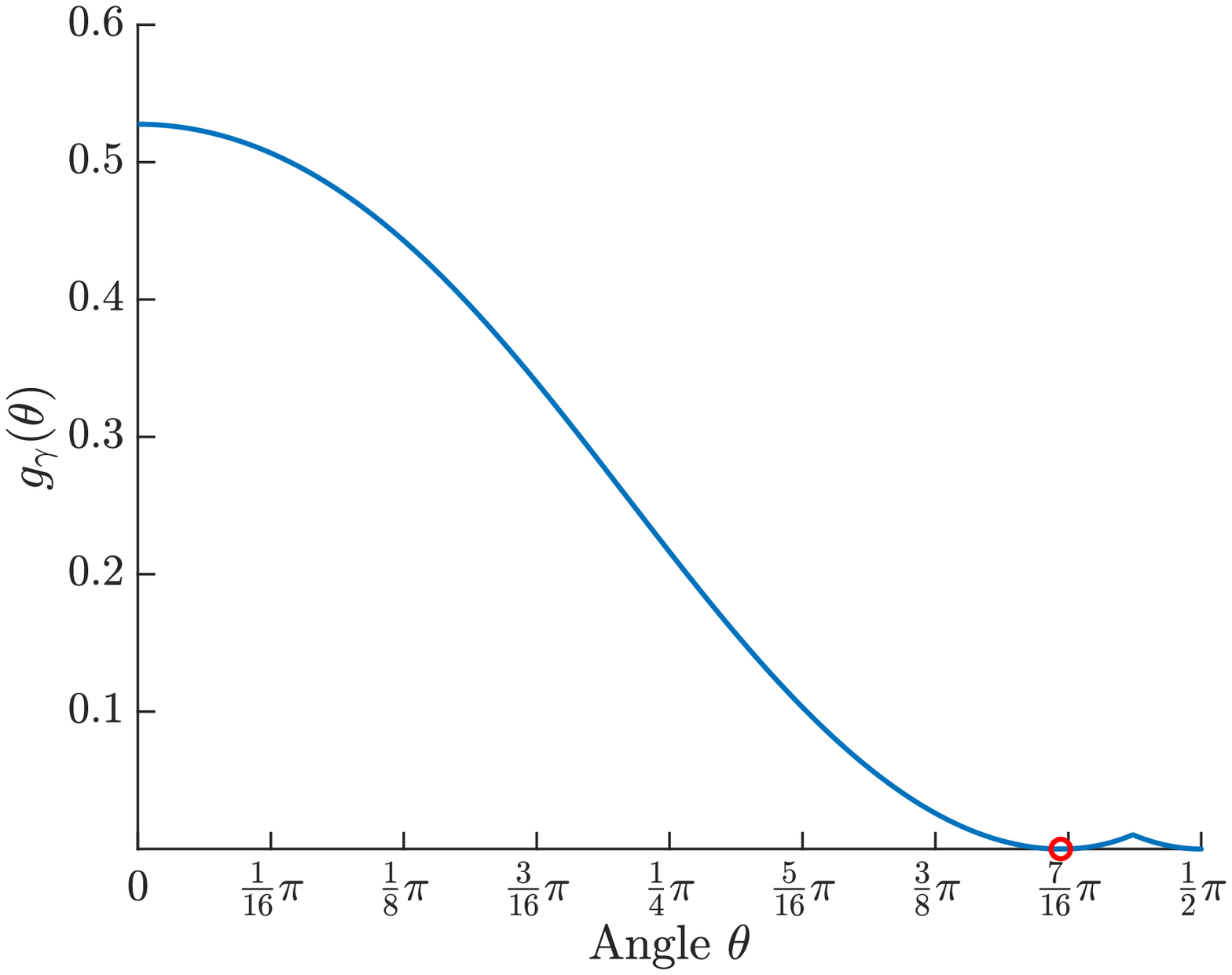} 
\includegraphics[scale=\scaledfn,trim={0.2cm 0.05cm 1.5cm 0.7cm},clip]{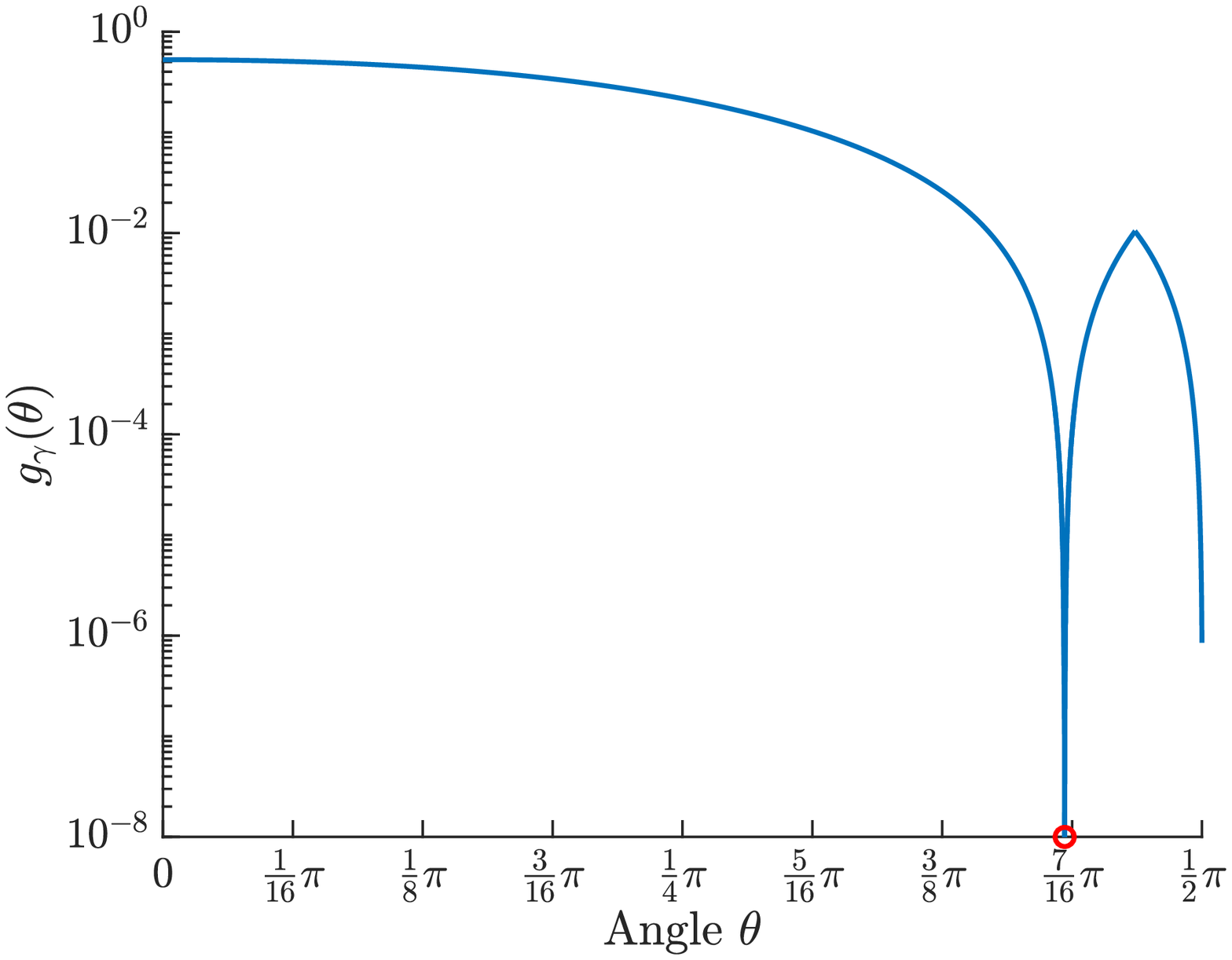} 
\label{fig:cert_boeing}
} 
\\
\subfloat[\texttt{randn} \#2: $\shfn(\theta)$ in linear scale (left) and in $\log_{10}$ scale (right)]{
\includegraphics[scale=\scaledfn,trim={0.3cm 0.05cm 0.6cm 0.7cm},clip]{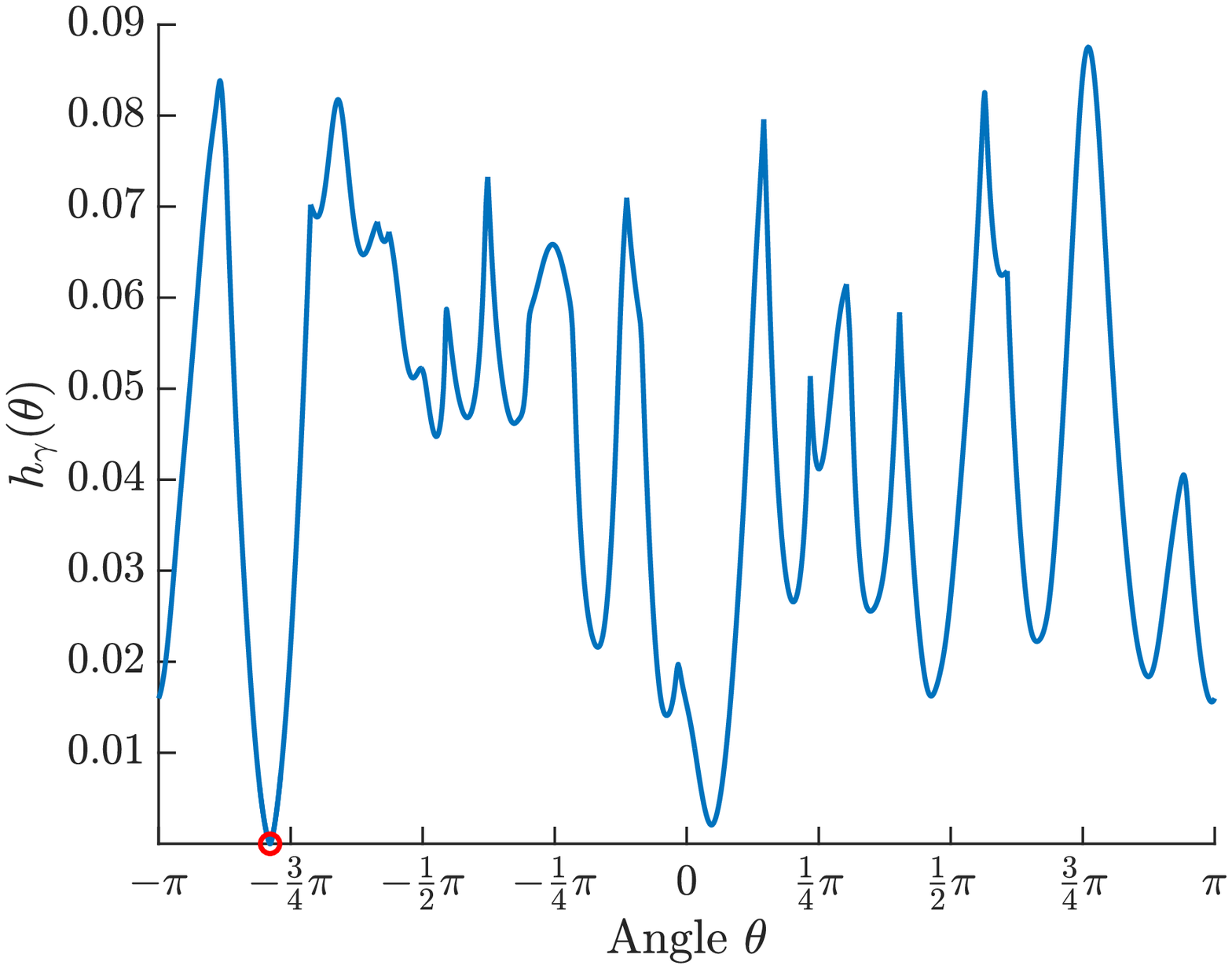} 
\includegraphics[scale=\scaledfn,trim={0.2cm 0.05cm 1.5cm 0.7cm},clip]{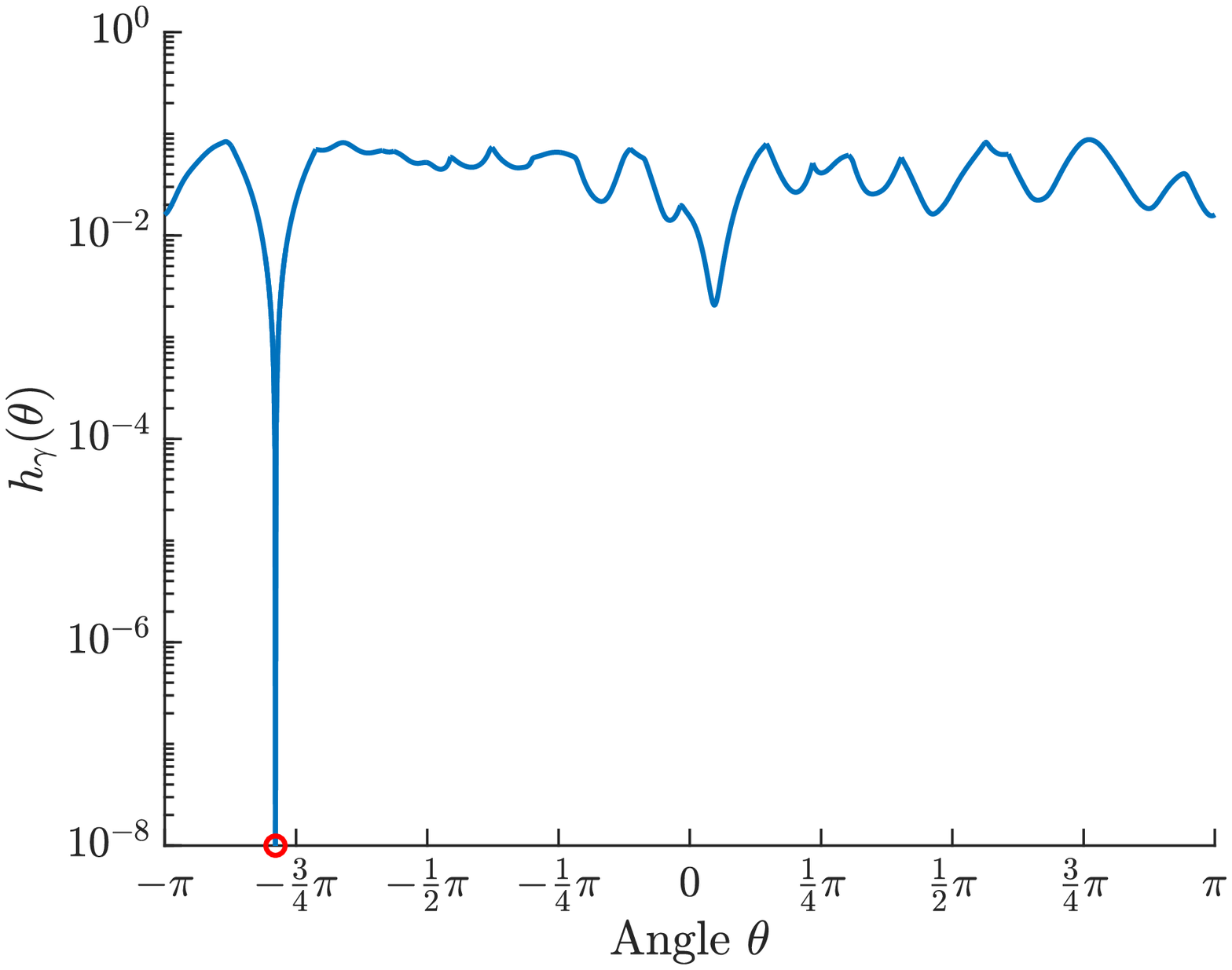}
\label{fig:cert_randn2}
}
\\
\subfloat[\texttt{kahan} ($m=30$): $\sffn(\theta)$ in linear scale (left) and in $\log_{10}$ scale (right)]{
\includegraphics[scale=\scaledfn,trim={0.3cm 0.05cm 0.6cm 0.7cm},clip]{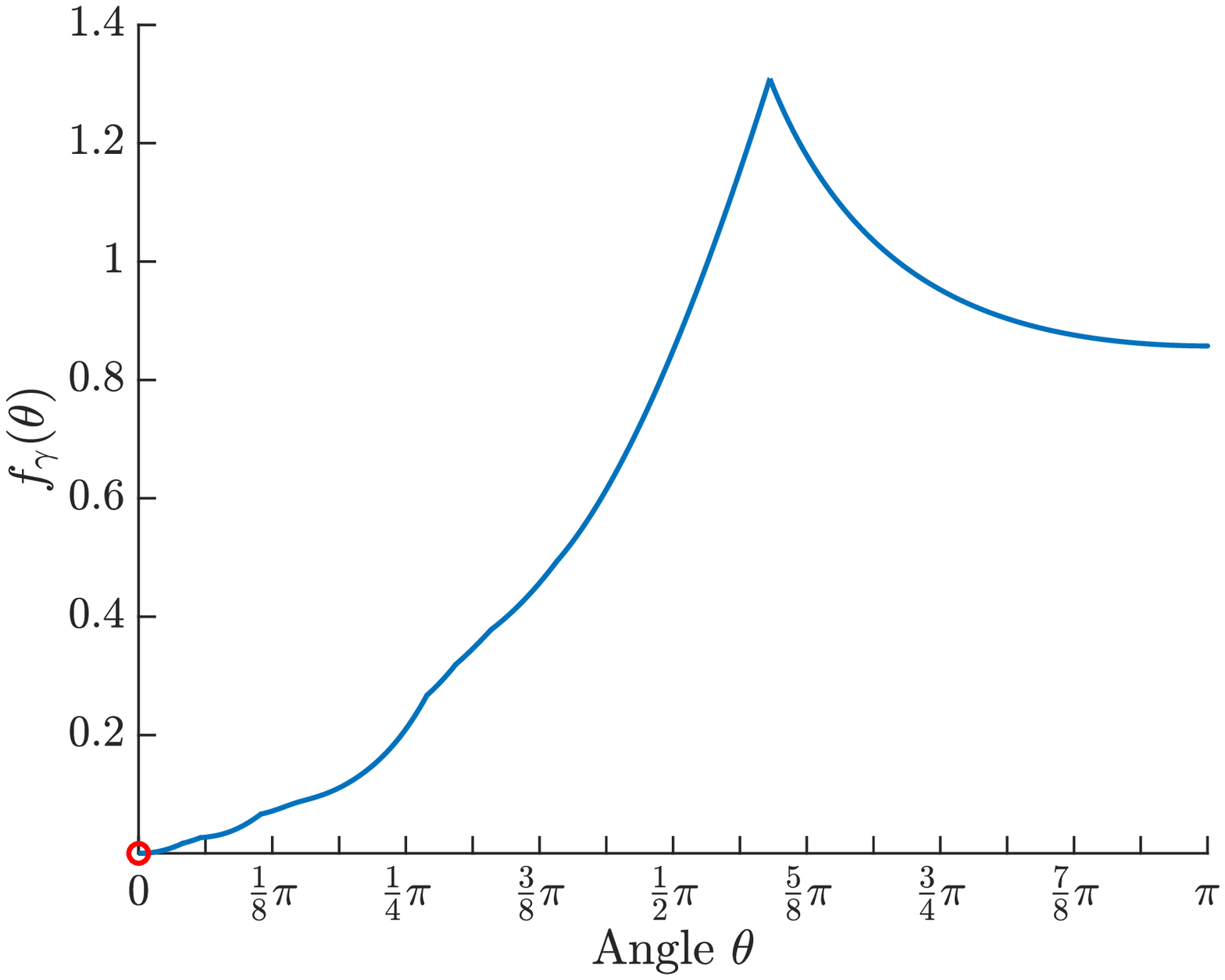} 
\includegraphics[scale=\scaledfn,trim={0.2cm 0.05cm 1.5cm 0.7cm},clip]{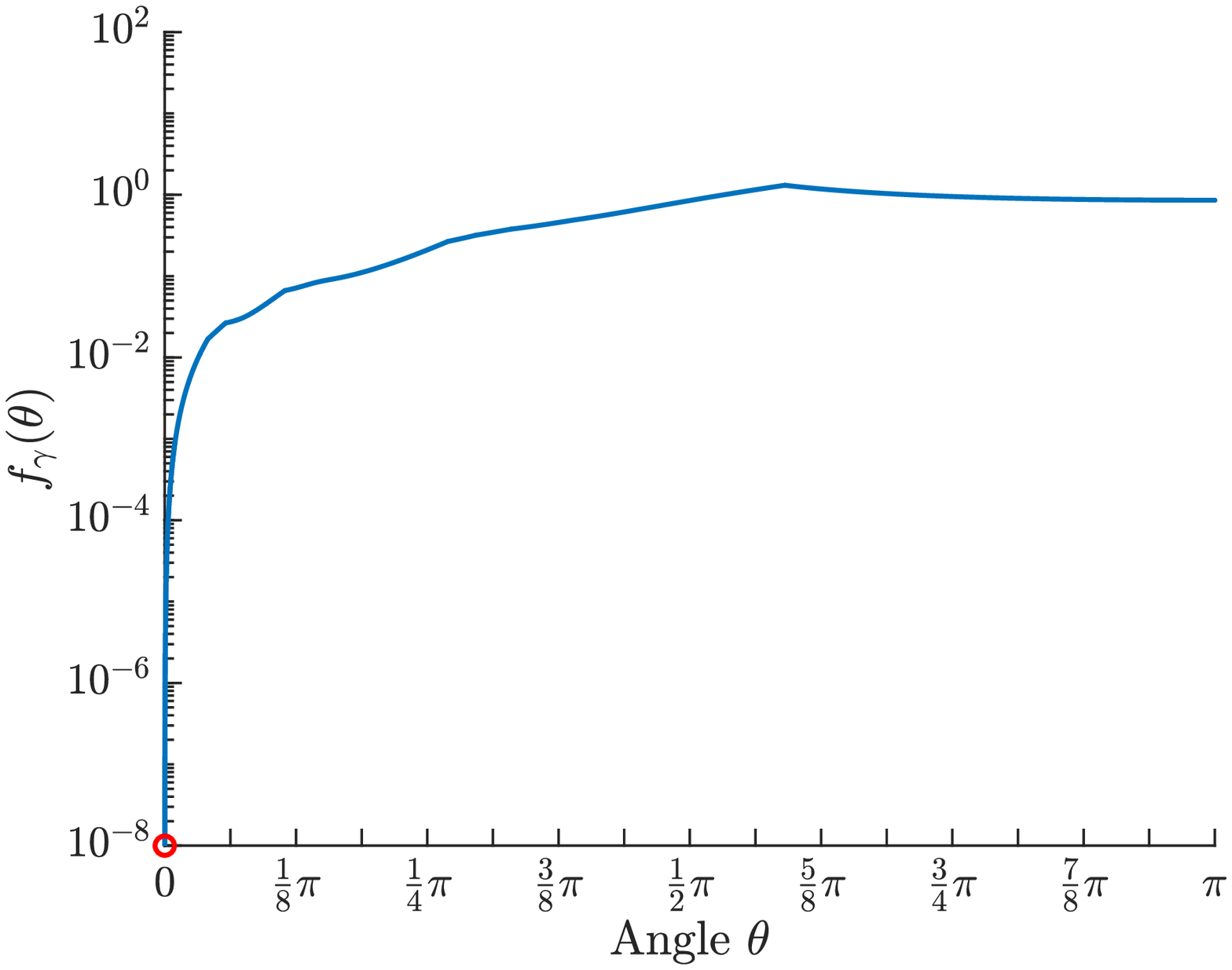} 
\label{fig:cert_kahan}
}
\caption{
The top two plots show $\sgfn(\theta)$ at the final value of $\gamma$ computed by our 
new method for the \texttt{boeing('S')} example, in linear and $\log_{10}$ scales.
The circle denotes the angle of the best minimizer obtained by optimization and corresponds to 
the single place where $\sgfn(\theta) = 0$ (which is more easily seen in the $\log_{10}$ plot on the right),
confirming that $\gamma$ is the globally minimal value.
The same is done for $\shfn(\theta)$ and \texttt{randn} \#2 in the middle plots
and for $\sffn(\theta)$ and \texttt{kahan} ($m=30$) in the bottom plots.
}
\label{fig:certs}
\end{figure}

\subsection{Additional acceleration via parallel processing}
\label{sec:par}
The main components of \cref{alg:interp} are ``embarrassingly parallel."
Optimization can be run from multiple starting points in parallel, to hopefully find
a global minimizer on any given iteration without increasing runtime.
For our interpolation-based globality certificates, any time Chebfun provides a vector of different values of $\theta$, 
obtaining the corresponding function values of $\sgfn(\theta)$, $\shfn(\theta)$, or $\sffn(\theta)$
is also ``embarrassingly parallel."
We focus on this latter task, as it is the dominant cost
and it is dependent on the average size of vectors provided by Chebfun.
To obtain speedup data, we recomputed the final certificates for our three largest problems, where \texttt{parpool(cores)} was called with $\texttt{cores}$ set to 2, 4, 8, 16, and 32, and the values of $\sgfn(\theta)$, $\shfn(\theta)$, and $\sffn(\theta)$ were computed inside a \texttt{parfor} loop.
We did these tests with the Chebfun preference \texttt{'min\_samples'} retained at its default value of 17
and with it increased to~65, comparing speedups with respect to our single-core configuration used in \S\ref{sec:comp}.

\begin{table}[t]
\centering
\small
\begin{tabular}{ l | ccccc  | rD{.}{.}{-1} } 
\toprule
\multicolumn{1}{c}{} & \multicolumn{5}{c}{Speedup per \# of cores} &  \multicolumn{2}{c}{Vector of $\theta$'s} \\
\cmidrule(lr){2-6}
\cmidrule(lr){7-8}
\multicolumn{1}{l}{Problem} & 
	\multicolumn{1}{c}{2} & 
	\multicolumn{1}{c}{4} & 	
	\multicolumn{1}{c}{8} & 
	\multicolumn{1}{c}{16} & 
	\multicolumn{1}{c}{32} & 
	\multicolumn{1}{c}{\#} & 
	\multicolumn{1}{c}{Avg. Size} \\
\midrule
\multicolumn{8}{l}{Chebfun \texttt{min\_samples}: 17} \\
\midrule
\texttt{orrsommerfeld}     &  2.4 &  3.6 &  4.9 &  6.4 &  5.7 &  184 & 16.6 \\ 
\texttt{randn} \#2 (stab.) &  2.9 &  4.5 &  6.3 &  8.5 &  9.0 &  797 & 21.8 \\ 
\texttt{kahan} ($m=30$)    &  2.6 &  3.9 &  5.5 &  7.7 &  8.6 &  430 & 14.5 \\ 
\midrule
\multicolumn{8}{l}{Chebfun \texttt{min\_samples}: 65} \\
\midrule
\texttt{orrsommerfeld}     &  2.6 &  3.8 &  5.3 &  6.6 &  6.1 &  162 & 20.1 \\ 
\texttt{randn} \#2 (stab.) &  3.0 &  4.9 &  6.9 &  9.1 &  8.8 &  608 & 29.3 \\ 
\texttt{kahan} ($m=30$)    &  2.7 &  4.3 &  6.2 &  8.3 &  8.8 &  385 & 17.3 \\ 
\bottomrule
\end{tabular}
\caption{Speedups with respect to the number of $\theta$'s evaluated per second 
while Chebfun is building the final interpolant for the three largest problems;
the reason speedups are not with respect to overall running time is because the total number of function evaluations Chebfun 
needed was not always the same as the number of cores was changed.
The last two columns, ``\#" and ``Avg.\ Size", respectively show the number of times Chebfun requested a vector of different 
values of $\theta$ to be evaluated and the average length of these vectors.
These average lengths give upper bounds on the best possible speedups, while
the pair of values together show that there is likely high overhead due to entering and exiting the \texttt{parfor} loop many times
in order for Chebfun to evaluate more and more points.
}
\label{table:parallel}
\end{table}

In \cref{table:parallel}, the best speedups range from 6.6 to 9.1, a significant boost.
While this is not high utilization of 32 cores,
the average number of $\theta$ values provided at a time by Chebfun,
which was often about 15 to 20, is an upper limit on achievable speedup.  
As reported in the \# column of \cref{table:parallel}, the parallel region of our code
is entered and exited hundreds of times, which comes with a very high overhead.  

Since varying \texttt{'min\_samples'} had little impact on performance and the total number of vectors, 
we analyzed the Chebfun code to determine how its amenability to parallelization might be improved. 
 Perhaps the biggest influence
is the \texttt{findJump} routine inside \texttt{@fun/detectEdge.m}, which does bisection
to detect singularities and thus requests only a single function value per iteration, for many iterations.
We modified \texttt{findJump} to instead do $k$-sectioning for integers $k > 2$ and found that our new version dramatically increased
the overall average vector length if $k$ was sufficiently large, 
as it also dramatically reduced the number of iterations \texttt{findJump} needed.
Another cause is related to the fact that Chebfun often approximates functions, particularly nonsmooth ones, 
not by a single polynomial interpolant but a concatenation of them.  
For each piece, a final safety test for accuracy (\texttt{@chebtech/sampleTest.m})
is done by evaluating a pair of \emph{hard-coded} points in the interval the piece is approximating over.
This too can keep the average vector length low and increase the total number of vectors.  
For parallel processing, it would be more efficient to 
speculatively evaluate these two fixed values for each piece, by
batching them in with first vector of initial sample points and storing this pair of function values for recall later.

\begin{remark}
Parallel eigensolvers such as \cite{BenKS14} could also be used to accelerate 
solving the large eigenvalue problems in the 2D level-set tests of \cite{Mit19} and \cite{Gu00},
but this would not reduce their high memory requirements nor does it seem likely that
this would be competitive with our interpolation-based certificates even using serial computation, let alone 
parallel computation.
\end{remark}

\section{Concluding remarks}
\label{sec:conclusion}
We have seen that our new interpolation-based globality certificates are generally orders of magnitudes more 
efficient than the existing techniques of \cite{Mit19} for Kreiss constants and 
those of \cite{Gu00,GuMOetal06} for the distance to uncontrollability.
While our new approach assumes that $\sgfn(\theta)$, $\shfn(\theta)$, and $\sffn(\theta)$ are adequately sampled to find their zeros,
this seems a rather mild assumption in practice, as per \cref{thm:props_kc,thm:props_kd,thm:props_dtu}, 
they will be zero on sets of \emph{positive measure} before a global minimizer has been obtained.  
The nature of our adequate interpolation assumption is quite different than the exact arithmetic assumption
used in the earlier methods of \cite{Gu00,GuMOetal06,Mit19}, and we believe it to be a more pragmatic choice,
both in terms of efficiency and reliability.
Finally, while in this paper we have considered the three specific problems of computing continuous-time Kreiss constants,
discrete-time Kreiss constants, and the distance to uncontrollability, we again emphasize 
our new approach of interpolation-based globality certificates
is for general global optimization problems of singular value functions in two real variables.
In fact, after submitting this manuscript, we have since used the idea of interpolation-based globality certificates 
to obtain a new algorithm for computing ``sep-lambda" \cite{Mit19b} that is much faster than the method of \cite{GuO06a}.
However, there are many fundamental differences in this case, 
both in the nature of the associated global optimization problem and our resulting algorithm.

\section*{Acknowledgments}
The author is very grateful to Michael L. Overton
for supporting several research visits to the Courant Institute in New York
and for many helpful comments on this manuscript.
The author also thanks the anonymous referees for reviewing the paper and their helpful feedback.

\bibliographystyle{alpha}
\bibliography{csc,software}

\end{document}